\documentclass{amsart}
\usepackage{amssymb}
\usepackage{latexsym}
\usepackage[margin=1.4in]{geometry}

\date{\today}

\newcommand{\Z}{{\mathbb Z}}
\newcommand{\R}{{\mathbb R}}

\newcommand{\C}{{\mathbb C}}

\newcommand{\PP}{{\mathbb P}}

\newcommand{\CA}{{\mathcal A}}

\newcommand{\CF}{{\mathcal F}}

\newcommand{\CH}{{\mathcal H}}

\newcommand{\CN}{{\mathcal N}}

\newcommand{\CU}{{\mathcal U}}

\def\SS{{\mathbb S}}

\def\t{{\theta}}
\def\l{{\lambda}}

\def\b{{\beta}}
\def\a{{\alpha}}
\def\e{{\varepsilon}}
\def\beq{\begin{equation}}
\def\eeq{\end{equation}}

\newcommand{\SL}{{\mathrm{SL}}}

\newtheorem{theorem}{Theorem}
\newtheorem{remark}{Remark}
\newtheorem{lemma}{Lemma}
\newtheorem{defi}{Definition}

\newtheorem{corollary}{Corollary}
\newtheorem{prop}{Proposition}
\begin{document}

\title[Uniform hyperbolicty and Spectral Theory]{Uniform hyperbolicity and its relation with spectral analysis of 1D discrete Schr\"odinger Operators}

\author[Z.\ Zhang]{Zhenghe Zhang}

\address{Department of Mathematics, University of California--Riverside, CA~92521, USA}

\email{zhenghe.zhang@ucr.edu}

\thanks{The author was in part supported by NSF grant DMS-1764154.}

\begin{abstract}
This paper intends to provide new, simple, and self-contained proofs of the equivalence of various different descriptions of the uniformly hyperbolic $\mathrm{SL}(2,\R)$ sequences. While in the scenario of the Schr\"odinger cocyles, they may in turn be applied to give new and simple proofs of theorems regarding their relation with the spectral analysis of one-dimensional discrete Schr\"odinger operators. Concretely, this paper gives four different descriptions of uniformly hyperbolic sequences together with the detailed proofs of their equivalence. It provides concise and self-contained proof of the Johnson's Theorem \cite{johnson}, both for sequence and dynamically defined potentials. In particular, we give two different proofs of the direction ``uniform hyperbolicity away from the spectrum'', of which one is the standard argument in the spirit of Russell Johnson's original proof and the other involves Combes-Thomas type of estimate \cite{combesthomas}. Finally, the relation between the Avalanche Principle, discovered by Goldstein-Schlag \cite{goldstein}, and the uniformly hyperbolic sequence is explored which yields a simple proof and a better version of the Avalanche Principle. Many ingredients of the proofs in this paper are new, in particular the use of asymptotic stable and unstable directions, are of independent interest and have been applied to many other related problems.
\end{abstract}

\maketitle

\section{Introduction}
The one dimensional discrete ergodic Schr\"odinger equation models the motion of a quantum particle in a disordered medium, such as alloys and quasicrystals. The medium is described by a bounded sequence of real numbers $v : \Z \rightarrow\R$, called the {\em potential}. The quantum particle is described by its wavefunction $\psi\in\ell^2(\Z)$, which is also called \emph{the state}, and evolves according to the famous Schr\"odinger equation
\beq\label{equations}
i\frac{\partial \psi}{\partial t}=H_{v}\psi.
\eeq
Here $H_{v}:\ell^2(\Z)\rightarrow\ell^2(\Z)$ is the so-called Schr\"odinger operator, or Hamiltonian, and is given by
\beq\label{eq:operators2}
(H_v \psi)_n=\psi_{n+1}+\psi_{n-1}+v(n)\psi_n,\ \psi=(\psi_n)_{n\in\Z}\in\ell^2(\Z).
\eeq
We assume $\|v\|_\infty<M$ for some $M>0$. Throughout this paper, $M$ serves as an upper bound in various different settings. The study of the solutions of equation~\eqref{equations} is closely related to the spectral analysis of the operator~\eqref{eq:operators2}. Note the spectrum, $\sigma(H_v)$, of the operator $H_v$ is:
$$
\sigma(H_v)=\{E\in\C: H_v-E \mbox{ is not invertible}\}.
$$
It is a standard result that $\sigma(H_v)$ is a compact set contained in $[-M-2,M+2]$ since $H_v$ is a bounded and self-adjoint operator with operator norm $\|H_v\|\le 2+M$. Let $\rho(H_v)=\R\setminus\sigma(H_v)$ denotes the resolvent set of $H_v$ on the real line. 

A key part of the spectral analysis of the operators \eqref{eq:operators2} is to understand the asymptotic behaviors of solutions of the spectral equation
\beq\label{eq:SpecEq}
H_v\psi=E\psi,
\eeq
where $E\in\C$ is the \emph{energy parameter}. A direct computation shows that $\psi\in\C^\Z$ solves equation~\eqref{eq:SpecEq} if and only if 
\beq\label{eq:SchrCocycle1}
A^{(E-v)}(j)\binom{\psi_j}{\psi_{j-1}}=\binom{\psi_{j+1}}{\psi_{j}}, \ j\in\Z,
\eeq
where $A^{(E-v)}:\Z\to\SL(2,\R)$ is called the \emph{Schr\"odinger cocycle map} and is defined as
\beq\label{eq:SchrCocylce2}
A^{(E-v)}(j)=\begin{pmatrix}E-v(j) &-1\\ 1 &0\end{pmatrix}.
\eeq
The cocycle iteration is defined as 
\beq\label{eq:cocycleiteration}
A^{(E-v)}_n(j)=\begin{cases}A^{(E-v)}(j+n-1)\cdots A^{(E-v)}(j), & n\ge1,\\ I_2 , & n=0, \\ A^{(E-v)}(j+n)^{-1}\cdots A^{(E-v)}(j-1)^{-1}, & n\le-1,\end{cases}
\eeq 
where $I_2$ is the identity matrix. By \eqref{eq:SchrCocycle1}, $A^{(E-v)}_n(j)$ are the $n$-step transfer matrices of the equation~\eqref{eq:SpecEq} since
\beq\label{eq:n_step_transfer0}
A^{(E-v)}_n(j)\binom{\psi_j}{\psi_{j-1}}=\binom{\psi_{j+n}}{\psi_{j+n-1}}\mbox{ for all } j, n\in\Z.
\eeq
Through this relation, the spectral analysis of the operator~\eqref{eq:operators2} may then be turned into the study of the dynamics of the cocycle iterations \eqref{eq:cocycleiteration}. 

This interplay of different areas, mathematical physics, spectral theory, and dynamical systems, has been a field of very active study since the late 1970’s. Moreover, it has made striking progress in the past 20 years since seminal works of e.g. Jitomirskaya \cite{jitomirskaya}, Bourgain-Goldstein \cite{bourgaingoldstein}, Goldstein-Schlag \cite{goldstein}, and Avila \cite{avila2, avila3}. For more information, we refer the readers to the book \cite{bourgain0} by Bourgain and recent surveys by Damanik \cite{damanik} and Jitomirskaya-Marx \cite{jitomirskayamarx}. 

The goal of this paper is to focus on a key notion, \emph{uniform hyperbolicity}, in dynamical systems and its relations to the spectral analysis of the operators~\eqref{eq:operators2}. Uniform hyperbolicity may be taken as a starting point to understand the dynamics behind the one-dimensional discrete Schr\"odinger operators.

The structure of the remaining part of this paper is as follows. In Section~\ref{ss:uhsequence}, we discuss uniformly hyperbolic $\SL(2,\R)$-sequences, uniformly hyperbolic $\SL(2,\R)$ cocycles defined on base dynamics, and the relation between them. Then we state our main theorems concerning equivalent conditions of uniform hyperbolicity. In Section~\ref{ss:johnson}, we introduce the Johnson's theorem for sequence potentials. Then we will introduce the relation between sequence potentials and dynamically defined potentials and state the more well-known version of the Johnson's theorem for dynamically defined potentials. In Section~\ref{ss:APNew}, we give our version of the Avalanche Principle, discuss its relation with uniformly hyperbolic sequence, and provide some historic remarks regarding its generalizations. 

In Section~\ref{s:uhsequence}, we prove the results stated in Section~\ref{ss:uhsequence}. It contains another description of uniform hyperbolicity. In Section~\ref{s:johnson}, we prove the results stated in Section~\ref{ss:johnson}.  Finally, in Section~\ref{s:AP}, we prove the theorem stated in Section~\ref{ss:APNew}.

\subsection{Uniformly Hyperbolic $\SL(2,\R)$-Sequences}\label{ss:uhsequence}
Consider a map $A:\Z\rightarrow\mathrm{SL}(2,\R)$ with $\|A(j)\|\le M$ for all $j\in\Z$. We again define $A_n(j)$ be as in \eqref{eq:cocycleiteration}. Let $B\cdot \t$ denotes the induced transformation of $B\in\mathrm{SL}(2,\R)$ acting on projective space $\R\PP^1=\R/(\pi\Z)\ni \t$, i.e. a line passing through the origin is identified with the angle between itself and the positive half part of a horizontal line. By $\vec v\in\t$, we mean a vector $\vec v$ in the line with direction $\t$. In particular, we let $\vec\t$ denotes an unit vector in the line of direction $\t$.
Throughout this paper, $C,\ c$ will be universal constants, where $C$ is large and $c$ is small, and $\|\cdot\|_\infty$ always denotes the usual supremum norm in various scenarios. We first give the following definition. 
\begin{defi}\label{d.uhsequence}
	We say that $A$ is \emph{uniformly hyperbolic} $(\CU\CH)$ if there are two maps 
	$$
	u,\ s:\Z\rightarrow\R\PP^1
	$$ 
	such that
	\begin{enumerate}
		\item $u,s$ are $A$--invariant in the sense that for all $j\in\Z$, it holds that
		$$
		A(j)\cdot u(j)=u(j+1) \mbox{ and } A(j)\cdot s(j)=s(j+1).
		$$
		\item there exists $C>0,\lambda>1$ such that $\|A_{-n}(j)\vec v\|,\|A_n(j)\vec w\|\leq
		C\lambda^{-n}$ for all $n\geq1$, all $j\in\Z$, and all unit vectors $\vec v\in
		u(j),\vec w\in s(j)$.
	\end{enumerate}
	Here $u$ is called the unstable direction and $s$ the stable direction of $A$.
\end{defi}
\begin{remark}\label{r.uneqs}
	In many literatures, it is explicitly stated in the definition of uniform hyperbolicity that $u(j)\neq s(j)$ for all $j\in\Z$, or equivalently $\R^2=s(j)\oplus u(j)$ where one may think of $\t\in\R\PP^1$ to be an one-dimensional subspace of $\R^2$. We leave it implicit as $u(j)\neq s(j)$ for all $j\in\Z$ is an immediate consequence of the conditions in Definition~\ref{d.uhsequence}. Indeed, if $u(j)=s(j)$ for some $j$, then by $A$-invariance $s(j)=u(j)$ for all $j\in\Z$. Then for an unit vector $\vec s(j)\in s(j)$, it holds for all $n\ge 1$ that 
	 $$
	 \|A_n(j)\vec s(j)\|<C\l^{-n}.
	 $$
	 And for an unit vector $\vec u(j+n)\in u(j+n)$, it holds for all $n\ge 1$ that
	 $$
	  \|A_{-n}(j+n)\vec u(j+n)\|<C\l^{-n}.
	 $$
	 By $A$-invariance, it must hold that $A_{-n}(j+n)\vec u(j+n)\in u(j)=s(j)$. Hence 
	 $$
	 \vec u(j):=\frac{A_{-n}(j+n)\vec u(j+n)}{\|A_{-n}(j+n)\vec u(j+n)\|}=\pm \vec s(j).
	 $$
	 So for all $n\ge 1$, we obtain
	 $$
	 \|A_n(j)\vec u(j)\|=\frac{\|\vec u(j+n)\|}{\|A_{-n}(j+n)\vec u(j+n)\|}> c\l^{n},
	 $$
which in turn implies for all $n\ge 1$ that
$$
c\l^{n}<C\l^{-n}.
$$
This is clearly not possible. Note for the argument above to work, we do not need the condition $\|A\|_\infty<M$. On the other hand, under the condtion $\|A\|_\infty<M$, Lemma~\ref{l.distanceus} says that Definition~\ref{d.uhsequence} actually implies that $|u(j)-s(j)|>\gamma>0$ for all $j\in\Z$, where the distance is in $\R\PP^1$. Similarly, this remark applies to Definition~\ref{d.uhcocycle}. 
	\end{remark}
From now on, $A\in\CU\CH$ means $A$ is uniformly hyperbolic. We have the following equivalent condition for $\CU\CH$ sequence, which is called the condition of \emph{uniformly exponential growth}.
\begin{theorem}\label{t.uiformeg}
	$A\in\CU\CH$ if and only if there exists $c>0,\lambda>1$ such that $A$ satisfies the following uniform exponential growth condition
	\beq\label{eq:ueg} 
	\|A_n(j)\|\ge c\lambda^{n}\ \mbox{for all }n\in\Z_+\mbox{ and all }j\in\Z.
	\eeq
\end{theorem}

Theorem~\ref{t.uiformeg} first appeared as the version for cocycles defined on certain base dynamical systems, see e.g. Yoccoz \cite[Proposition 2]{yoccoz} or Viana~\cite[Proposition 2.1]{viana}. We also formulate the dynamical version as Corollary~\ref{c:uniformeg}, the proof of which follows from exactly the same argument of the proof of Theorem~\ref{t.uiformeg}. 

Consider a set $\Omega$, a bijection map $T:\Omega\rightarrow\Omega$, and a map $A:\Omega\rightarrow \mathrm{SL}(2,\R)$ with $\|A\|_\infty<M$. Then we may define a dynamical system
\beq\label{eq:cocycle0}
(T,A):\Omega\times\R^2\rightarrow\Omega\times\R^2\ (T,A)(\omega,\vec v)=(T(\omega),A(\omega)\vec v).
\eeq
Let $(T^n,A_n)=(T,A)^n$ denotes the iteration of the map. Then similar to \eqref{eq:cocycleiteration}, we have
\beq\label{eq:cocycleiteration1}
A_n(\omega)=\begin{cases}A(T^{n-1}\omega)\cdots A(\omega), & n\ge1,\\ I_2 , & n=0, \\  [A_{-n}(T^{n}\omega)]^{-1}& n\le-1,\end{cases}
\eeq 

Here $A$ is called a cocycle map. For simplicity, $(T,A)$ may also denote the induced projective dynamics of $(T,A)$ on $\Omega\times\R\PP^1$.

\begin{defi}\label{d.uhcocycle}
	$(T,A)$ is said to be uniformly hyperbolic if there exist two maps $u,\ s:\Omega\rightarrow\R\PP^1$ such that:
	\begin{enumerate}
		\item $u,s$ are $(T,A)$--invariant which means that for all $\omega\in\Omega$,
		$$
		A(\omega)\cdot u(\omega)=u[T(\omega)]\mbox{ and } A(\omega)\cdot s(\omega)=s[T(\omega)];
		$$
		\item there exists $C>0,\lambda>1$ such that $\|A_{-n}(\omega)\vec v\|,\|A_n(\omega)\vec w\|\leq
		C\lambda^{-n}$ for all $n\geq1$, all $\omega\in\Omega$, and all unit vectors $\vec v\in
		u(\omega),\vec w\in s(\omega)$
	\end{enumerate}
	Here $u$ is called the unstable direction and $s$ the stable direction of $(T,A)$.
\end{defi}
Note in Definition~\ref{d.uhcocycle}, no topological structure or $\sigma$-algebra structure is assumed for $\Omega$. Then we have the following corollary of the proof of Theorem~\ref{t.uiformeg}, see e.g. Remark~\ref{r:uniformeg}.

\begin{corollary}\label{c:uniformeg}
	Let $(\Omega,T,A)$ be as in Definition~\ref{d.uhcocycle}. Assume in addition $\Omega$ is a compact topological space, $T$ a homeomorphism and $A$ continuous. Then $(T,A)\in\CU\CH$ if and only if there exists $c>0,\lambda>1$ such that $(T,A)$ satisfies the following uniform exponential growth condition:
	$$
	\|A_n(\omega)\|\ge c\lambda^{n} \mbox{ for all }n\in\Z_+\mbox{ and all }\omega\in\Omega.
	$$
	Moreover, the corresponding unstable and stable directions are continuous on $\Omega$.
\end{corollary}
\begin{remark}\label{r:nuh}
	There is another notion of hyperbolicity that is closely related to, but different from, $\CU\CH$ which is called \emph{non-uniform hyperbolicity} ($\CN\CU\CH$). To define it, we have to introduce a $T-invariant$ probability measure $\mu$ on $\Omega$, i.e. a probability $\mu$ such that $\mu(T^{-1}(S))=\mu (S)$ for any $\mu$-measurable set $S\subset\Omega$. Then we introduce the dynamical object, the Lyapunov exponent, which is defined as
	$$
	L(T,A)=\lim_{n\to\infty}\frac1n\int_{\Omega}\log\|A_n(\omega)\|d\mu=\inf_{n\ge 1}\frac1n\int_{\Omega}\log\|A_n(\omega)\|d\mu\ge 0.
	$$
	The limit exists and is equal to the infimum since $\{\int_{\Omega}\log\|A_n(\omega)\|d\mu\}_{n\ge1}$ is subadditive. It is clear that if $(T,A)\in\CU\CH$, then $L(T,A)>\log\l>0$. On the other hand, if $L(T,A)>0$ and $(T,A)\notin\CU\CH$, then we say that $(T,A)$ is \emph{non-uniformly hyperbolic} and is denoted as $(T,A)\in\CN\CU\CA$. There is well-developed theory in dynamical systems called the \emph{Oseledec's Multiplicity Ergodic Theorem} which gurantees the existence of a pair of measurable stable and unstable direction $s$ and $u:\Omega\to\R\PP^1$ for any such system with positive Lyapunov exponent. See e.g. \cite{barreirapesin, viana}. One may also deduce the existence of such directions by following the proof of Lemma~\ref{l.existenceus}. The main difference between $\CU\CH$ and $\CN\CU\CH$ is that for $\CN\CU\CH$, the stable and unstable directions are only defined $\mu$ almost everywhere, are merely measurable, and their difference must tend to $0$ along some orbits.
	\end{remark}

One may go from a sequence to a dynamical system by the following process. First we need the notion of hull of a sequence in a space of full shift. Let $\CA^\Z$ be the space of full shift generated by a set of alphbets $\CA$. Suppose $\CA$ is a compact topological space and $\CA^\Z$ be equipped with the product topology. Hence $\CA^\Z$ is compact topologic space as well. Moreover, if $\CA$ is a metric space, then so is $\CA^\Z$. Let $T:\Omega\to\Omega$ be the operator of left shift, i.e.
 $$
 (T\omega)_n=\omega_{n+1} \mbox{ for } \omega=(\omega_n)\in\CA^\Z.
 $$
\begin{defi}
For each $\omega\in\CA^\Z$, the hull of $\omega$ is defined as $\overline {\{T^n(\omega)\}_{n\in\Z}}$, i.e. the closure of the $T$-orbit of $\omega$ under the product topology. Let $\mathrm{Hull}(\omega)$ denotes the hull of $\omega$, which itself is clearly a compact topological space that is invariant under $T$.
\end{defi}
Now we take $\CA=B_M[\mathrm{SL}(2,\R)]$ where $B_M$ denotes the ball in $\mathrm{SL}(2,\R)$ with norm less than or equal to $M$. Thus the map $A:\Z\to B_M[\mathrm{SL}(2,\R)]$ is an element in $\CA^\Z$. So we may let $\Omega=\overline{\{T^n(A)\}_{n\in\Z}}=\mathrm{Hull}(A)$. Clearly, $T:\Omega\rightarrow\Omega$ is a homeomorphism. Let $F:\Omega\rightarrow\mathrm{SL}(2,\R)$ be the evaluation map at the $0$-position, i.e. $F(\omega)=\omega(0)$. We may then consider the cocycle $(T,F):\Omega\times\R^2\to\Omega\times\R^2$ as in \eqref{eq:cocycle0}. Let $(T^n,F_n)=(T,F)^n$ with $F_n(\omega)=F(\omega_{n-1})\cdots F(\omega_0)$ be the cocycle iteration as in \eqref{eq:cocycleiteration1}. Then the following proposition is straightforward:

\begin{prop}\label{c:uniformeg2}
	  Let $A$ and $(\Omega, T, F)$ be as above. Then 
	  $$
	  \|A_n(j)\|\ge c\l^{n}  \mbox{ for all } n \in\Z_+ \mbox{ and all } j\in\Z
	  $$ 
	  if and only if 
	  $$
	  \|F_n(\omega)\|\ge c\l^{n} \mbox{ for all } n \in\Z_+ \mbox{ and all } \omega\in\Omega.
	  $$
	  In particular, by Therorem~\ref{t.uiformeg}, the sequence $A$ is uniformly hyperbolic if and only if $(T,F)$ is uniformly hyperbolic.
\end{prop}
\begin{proof}
We only need to consider the \emph{only if} part since the \emph{if} part is obvious via the relation $F_n(T^kA)=A_n(k)$. Fix any $\omega\in\Omega$ and any $n\in\Z$. Since $\{T^k(A), k\in\Z\}$ is dense in $\Omega$, for each $\varepsilon>0$, we can find a $j$ that $T^j(A)$ is so close to $\omega$ that the following holds:
$$
\|F_n(\omega)\|>(1-\varepsilon) c\l^{n}.
$$ 
Since the above inequality holds for all $\omega\in\Omega$,  all $n\in\Z$ and all $\omega>0$, we then get 
 $$
\|F_n(\omega)\|\ge c\l^{n} \mbox{ for all } n \in\Z_+ \mbox{ and all } \omega\in\Omega,
$$
concluding the proof.
	\end{proof}
It turns  out that for the proof of Theorem~\ref{t.uiformeg}, one naturally needs to move from a sequence to its Hull. Next we have another equivalent description of uniform hyperbolicity which is used for the proofs in Section~\ref{s:johnson}. Let $\Omega$ be compact metric space\footnote{Unlike Corollary~\ref{c:uniformeg}, here we need $\Omega$ to be a compact metric space as we will need sequential compactness in its proof.}, $T:\Omega\to\Omega$ a homeomorphism, and $A:\Omega\to\mathrm{SL}(2,\R)$ is continuous. 
\begin{theorem}\label{l.nontrivialbo}
	Let $\Omega, T, A$ be as above mentioned. In particular, $\Omega$ is a compact metric space. Then $(T,A)\notin\CU\CH$ if and only if there is a $\omega\in\Omega$ and an unit vector $\vec v\in\R^2$ such that
	\beq\label{eq:NonTrivailBounded}
	\|A_n(\omega)\vec v\|\le1,\mbox{ for all } n\in\Z.
	\eeq
\end{theorem}
Theorem~\ref{l.nontrivialbo} clearly implies the following equivalent description of uniform hyperbolic sequence.
\begin{corollary}\label{c.nontrivialbo}
	The sequence $A:\Z\to\SL(2,\R)$ is not uniformly hyperbolic if and only there is an unit vector $\vec v\in\R^2$ so that the following holds true. For each $\e>0$ and each $n\in\Z$, there is a $j_n\in\Z$ so that
	$$
	\|A_n(j_n)\vec v\|<1+\e.
	$$
	\end{corollary}
\begin{proof}
	Let $\Omega=\mathrm{Hull}(A)$. Let $(T,F)$ be the dynamics on $\Omega\times\R^2$ as before Proposition~\ref{c:uniformeg2}. Then by Proposition~\ref{c:uniformeg2}, $A\notin\CU\CH$ if and only if $(T,F)\notin\CU\CH$. By Theorem~\ref{l.nontrivialbo}, $A\notin\CU\CH$ if and only if there is a $\omega\in\Omega$, an unit vector $\vec v\in\R^2$ such that
	$$
	\|F_n(\omega)\vec v\|\le1,\mbox{ for all } n\in\Z.
	$$
	Since $\{A(j)\}_{j\in\Z}$ is dense in $\Omega$ and $F_n(T^kA)=A_n(k)$, by a standard continuity argument, we then obtain the desired conclusion.
	\end{proof}
One can basically find Theorem~\ref{l.nontrivialbo} in \cite{sackersell}, see for example \cite[Theorem 1.7]{johnson}. The author learned the proof of Theorem~\ref{l.nontrivialbo} presented in Section~\ref{s:uhsequence} from a course given by Artur Avila. It is perhaps relatively simple. See Remark~\ref{r:johnson} for some historic remarks regarding Theorem~\ref{l.nontrivialbo}.

Another equivalent condition for uniform hyperbolicty is the existence of an invariant cone field, see e.g. \cite[Section~2.1]{avila} for some detailed description. Though we do not involve the equivalence between it and $\CU\CH$ directly anywhere, in Section~\ref{ss:uhsequence2}, we do make use of a version of the invariant cone field. See Lemma~\ref{l:inv_cone} for more information.

\subsection{Resolvent Set and Uniform Hyperbolicity}\label{ss:johnson}

 One of the important relations between dynamics of the Schr\"odinger cocycle and spectral theory of the Schr\"odinger cocyle is the following Johnson's Theorem.
	\begin{theorem}\label{t.uniformhers}
		$\sigma(H_v)=\{E\in\R:A^{(E-v)}\notin\CU\CH\}.$
	\end{theorem}
\begin{remark}\label{r:johnson}
	Johnson's Theorem was first presented in Russell Johnson's paper \cite{johnson}. It was originally stated in the setting of dynamically defined Schr\"odinger operators (See e.g. the Theorem~\ref{t.uniformhers2} below). Johnson's work used and was partially inspired by series of works of Sacker-Sell, see e.g. \cite{sackersell}. In particular, Sacker-Sell discovered the so-called Sacker-Sell orbit which is essentially the $(A_n(\omega)\vec v)_{n\in\Z}$ as in \eqref{eq:NonTrivailBounded}. Note also that in Sacker-Sell and Johnson's papers the notion of uniform hyperbolicity are appeared as the notion of \emph{exponential dicotomy}. We also wish to point out that separately in the community of dynamical systems, diffeomorphisms on manifolds which do not admit Sacker-Sell type of orbits on the tagent bundle are called quasi-Anosov, which have been heavily studied as well. For diffeomorphisms on manifold, the equivalent notion of uniformly hyperbolic system is called the Anosov diffeomphism. See e.g. the paper by Franks-Robinson \cite{franks} for some further information.
	\end{remark}

It turns out that like the uniform hyperbolic sequences, even though one starts with the operator \eqref{eq:operators2}, one naturally ends up studying a family of operators. More precisely, in this case, we let $\CA=[-M,M]$ with the usual topoloty. Hence, $\CA^\Z$ is compact under product topology. Let $\Omega=\mathrm{Hull}(\CA)=\overline{\{T^n(v)\}_{n\in\Z}}$ which is a compact topological space that is invariant under the left shift operator $T$. We may then define a funtion $f:\Omega\to\R$ as $f(\omega)=\omega_0$ and consider the family of the operators
\beq\label{eq:operator}
(H_\omega \psi)_n=\psi_{n+1}+\psi_{n-1}+f(T^n\omega)\psi_n,\ \omega\in\Omega.
\eeq
In particular, $H_v$ is embedded into the dynamically defined family of operators $\{H_\omega\}_{\omega\in\Omega}$. Note that $T$ is topological transitive since $\{T^nv, n\in\Z\}$ is dense in $\Omega$. For spectral analysis of the family of operators $\{H_\omega\}_{\omega\in\Omega}$, $(\Omega, T)$ usually comes equipped with a $T$-ergodic probability measure $\mu$. Here ergodic measure means that $\mu$ is a $T$-invariant measure with the additional property:
\begin{center}
$\mu\left[(A\setminus T^{-1}A)\bigcup (T^{-1}A\setminus A)\right]=0\Rightarrow\mu(A)=0$ or $1$. 
\end{center}
Then the family of operators $\{H_\omega\}_{\omega\in\Omega}$ is called ergodic Schr\"odinger operators with ergodic base dynamics $(\Omega,T, \mu)$. To study many of the spectral properties of the operator~\eqref{eq:operators2}, such as Anderson Localization phenomenon, one has to consider the family operators~\eqref{eq:operator} for $\mu$-a.e. $\omega$ with some suitable choice of $\mu$. See Remark~\ref{r:le} for more information.

In fact, one may just start with ergodic system $(\Omega,T,\mu)$ where $\Omega$ is a compact metric space, $T:\Omega\to\Omega$ a homeomorphism and $\mu$ is a probability on $\Omega$ that is $T$-ergodic. Let $f:\Omega\to\R$ be a continuous function. Then we can define $H_\omega$ just as in \eqref{eq:operator}. Similarly, the transfer matrices may be generated as follows. First, we define a family maps $A^{(E-f)}:\Omega\to\SL(2,\R)$ as
\beq\label{eq:schrodinger_cocycle_map}
A^{(E-f)}(\omega)=\begin{pmatrix}E-f(\omega)&-1\\ 1&0\end{pmatrix},\ E\in\R.
\eeq
which is called the \emph{Schr\"odinger cocycle map}. Then similar to \eqref{eq:cocycle0}, we have a family of dynamical systems as
\beq\label{eq:schrodinger_cocycle}
(T,A^{(E-f)}):\Omega\times\R^2\rightarrow\Omega\times\R^2,\ (T,A^{(E-f)})(\omega,\vec v)=(T\omega,A^{(E-f)}(\omega)\vec v)
\eeq
which is called the \emph{Schr\"odinger cocycle}. Let $(T,A^{(E-f)})^n=(T^n,A^{(E-f)}_n)$. Similar to \eqref{eq:n_step_transfer0}, $A^{(E-f)}_n(\omega)$ is the $n$-step transfer matrix of the equation $H_\omega\psi=E\psi$. Then we have the following theorem which is the standard version of the Johnson's Theorem \cite{johnson}:

\begin{theorem}\label{t.uniformhers2}
	Let $(\Omega,T, f)$ be as above mentioned. Assume in addition that $(\Omega,T)$ is topological transitive. Let $\omega_0$ be that $\overline{\mathrm{Orb}(\omega_0)}=\Omega$ and let $\Sigma=\sigma(H_{\omega_0})$. Then it holds for all $\omega\in\Omega$ that $\sigma(H_\omega)\subset\Sigma$. Moreover, it holds that
	$$
	\Sigma=\{E:(T,A^{(E-f)})\mbox{ is not uniformly hyperbolic}\}.
	$$
\end{theorem}

\begin{remark}\label{r:le}
	Let $L(E)=L(T,A^{(E-f)})$ be the Lyapunov exponent for each $E\in\R$. By Theorem~\ref{t.uniformhers2} and Remark~\ref{r:nuh}, for those $E\in\Sigma$, either $L(E)=0$ or $(T,A^{(E-f)})\in\CN\CU\CH$. For the spectral analysis of ergodic Schr\"odinger operators, a key part is to study the Lyapuonv exponent $L(E)$. For instance, Kotani Theory \cite{kotani} basically identifies for $\mu-$a.e. $\omega$ the absolutely spectrum $H_\omega$ with the set of energies where $L(E)=0$. Moreover, the approach developed by Bourgain-Goldstein \cite{bourgaingoldstein} basically shows that uniform positivity and some version of uniform large deviation estimates of the Lyapunov exponent $L(E)$ (both uniformities are in $E$) are strong indications of Anderson Localization for $H_\omega$ for many $\omega$'s. Also, a key part of the global theory established by Avila \cite{avila2} for one-frequency quasiperiodic operators are some deep analysis of the Lypunov exponent. In particular, Avila managed to characterize $\CU\CH$ and $\CN\CU\CH$ via the Lyapunov exponent in a surprising way for analytic one-frequency $\mathrm{SL}(2,\C)$ cocyles.
	
	\end{remark}

We wish to point out that Theorem~\ref{t.uniformhers} and ~\ref{t.uniformhers2} are closely related due to the relation between $v$ and $\mathrm{Hull}(v)$. However, Theorem~\ref{t.uniformhers} is the first version for sequence potentials. It thus may come more naturally. In fact, in Section~\ref{s:johnson}, we could deduce Theorem~\ref{t.uniformhers2} from Theorem~\ref{t.uniformhers} without much effort. Moreover, in Section~\ref{s:johnson}, we shall give two different proofs for the direction ``uniform hyperbolicity away from the spectrum''. The first one uses the results in Section~\ref{ss:uhsequence}. The other one is provided to the author by W. Schlag which only uses Combes-Thomas type of estimates and is actually much simpler.

Since the work of Johnson, there are many further developments. For instance, \cite{marx} showed a similar correspondence between the singular Jacobi operators and cocycles with dominated splitting, which is a generalized notion of $\CU\CH$ where matrices are allowed to be singular. Later, \cite{fillmanongzhang} did a similar version between the generalized extended CMV matrices and cocyles with dominated splitting. Those are crucial in establishing many spectral properties of the extended Harper's model as in \cite{avilajitomirskayamarx} and the unitary critical almost Mathieu operators as in \cite{fillmanongzhang}. See also \cite{gorodetskikleptsyn} for a recent version written in terms of the rotation number of the cocycle.

\subsection{Avalanche Principle and Uniform Hyperbolicity}\label{ss:APNew}
The following proposition was first discovered by Goldstein-Schlag \cite{goldstein} and is called \emph{the Avalanche Principle}. 
\begin{prop}\label{p:avlanche-principle}
	Let $A(1),\ldots, A(n)$ be a finite sequence in $\mathrm{SL}(2,\R)$ satisfying:
	\begin{align}\label{condition-AP}
	&\min_{1\le j\le n}\|A(j)\|\ge \l > n,\\ \label{condition-AP2}
	&\max_{1\le j<n}\left|\log \|A(j+1)\|+\log \|A(j)\|-\log\|A(j+1)A(j)\|\right|<\frac12\log\l.
	\end{align}
	Then
	\begin{equation} \label{avalanche-principle}
	\left|\log\|A_n(1)\|+\sum_{j=2}^{n-1}\log\|A(j)\|-\sum_{j=1}^{n-1}\log\|A(j+1)A(j)\|\right|\le C\frac{n}{\l}.
	\end{equation}
\end{prop}
\noindent See \cite[Proposition 2.2]{goldstein} for the original proof. Basically, the Avalanche Principle permits us good control on the norm of a product of $\SL(2,\R)$ matrices provided we have suitable estimates on consecutive pairwise products. Together with large deviation type of estimates for the associated Lyapuonv exponent, it is proved to be a powerful tool in establishing quantative continuity of the Lyapunov exponent and the integrated density of states.

A close comparision of the conditions \eqref{eq:largenorm}-\eqref{eq:largedistance} of Corollary~\ref{c:SeparationUStoUH} and the conditions~\eqref{condition-AP}-\eqref{condition-AP2}, one may find they are basically conditions~\eqref{eq:largenorm}-\eqref{eq:largedistance} for $k=1$. Thus the finite sequence $A(1),\ldots, A(n)$ is actually a finite piece of an infinitly hyperbolic sequence. A bit more precisely, if $\|A(j+1)A(j)\|$ is not too small compared with $\|A(j+1)\|\cdot\|A(j)\|$, then the most contracted direction $s(j+1)$ of $A(j+1)$ is not too close to the most contracted direction $u(j)$ of $(A(j))^{-1}$. 

Once one realizes this relation, then one can actually remove the restriction $\l>n$ in condition~\eqref{condition-AP} and obtain a dynamical proof. Roughly speaking, in case one has $\CU\CH$, as the norm of the cocycle satisfies the condition of uniform exponential growth \eqref{eq:ueg}, the difference between $n$-step asymptotic stable (resp. unstable) and $(n+k)$-step asymptotic stable (resp. unstable) directions (see \eqref{eq:AsympDirections} for their definition) won't accumulate as $k$ gets large, see \eqref{eq:exp_decay_error_s_n}. One may compare the proof of Lemma~\ref{l.existenceus} and Section~\ref{s:AP} for more information of this discussion. So we restate it as: 
\begin{theorem}\label{t:APNew}
	\label{l.avlanche-principle}
	Let $A(j),\ j\in\Z$ be an infinite sequence in $\mathrm{SL}(2,\R)$. Suppose there is a $\l>C$, largeness independent of $n$ below, so that for each $j$, it holds that:
	\begin{align}\label{condition-AP3}
	&\|A(j)\|\ge \l, \\ \label{condition-AP4}
	&\left|\log \|A(j+1)\|+\log \|A(j)\|-\log\|A(j+1)A(j)\|\right|\le \frac12\log\l.
	\end{align}
	Then the sequence $A(j),\ j\in\Z$ is uniformly hyperbolic and it holds for each $j\in\Z$ and each $n\in\Z_+$ that
	\begin{equation}\label{eq:AP}
	\left|\log\|A_n(j)\|+\sum_{k=1}^{n-2}\log\|A(j+k)\|-\sum_{k=0}^{n-2}\log\|A(j+k+1)A(j+k)\|\right|\le C\frac{n}{\l}.
	\end{equation}
	
\end{theorem}

There are numerous generalizations since the original work of Goldstein-Schlag, e.g. orders of the matrices have been generalized from $2$ to any $d\ge 2$, real valued matrices to complex valued ones, and the relaxation of the condition $\l>n$. In particular, in the setting of $\SL(2,\R)$-cocycles, as it was expressed by Duarte-Klein \cite{duarteklein}, Bourgain-Jitomirskaya \cite[Lemma 5]{bourgainjitomirskaya} have greatly relaxed the constraint $\l>n$, and later Bourgain \cite[Lemma 2.6]{bourgain} has completely removed it, at the cost of slightly weakening the conclusion of the AP. Finally,  Duarte-Klein's work \cite{duarteklein} may be considered as a one stop reference for the development of Avalanche Principle, where they obtained the most general version. In particular, in their version, $\l>n$ is removed without any cost. 

Our purposes of Theorem~\ref{t:APNew} are the following. First, no version above mentioned seemed to realize the relation between infinitely uniformly hyperbolic sequence of matrices and the Avalanche Principle. Second, though the version in \cite{duarteklein} is more general than ours (note however, they only dealt with finite sequences), the proof was quite involved and hence much longer than the one in this paper. 

The proof of Theorem~\ref{t:APNew} is the whole Section~\ref{s:AP}. Many of the tools have actually been developed in \cite{wangzhang} for different purpose concerning positivity and large deviation estimates of the Lyapunov exponent of some quasiperiodic operators. We include a full proof here as the tools are a special case of those in \cite{wangzhang} and hence the proof are much simpler. From the proof, it is not difficult to see that one may both improve the estimates and generalize the results in various ways. However, we wish to stay within the simplest nontrivial scenario to make the dynamics behind the Avalanche Principle clearest.

\section{Equivalent Descriptions of Uniform hyperbolicity} \label{s:uhsequence}

In this section, we prove Theorems~\ref{t.uiformeg} and \ref{l.nontrivialbo}. It is clear that Definition~\ref{d.uhsequence} implies the uniform exponential growth of $A$. So we only need to show that the converse is true. 

\subsection{Uniform Exponential Growth Implies Uniform Hyperbolicity}\label{ss:uhsequence2} \footnote{The author would like to thank W. Schlag for pointing out to him that there are some similar discussions in \cite[Section 2.2]{viana} with the present section. Corollary~\ref{c:cont_su} was added after the author read Viana's book.}

For $D\in\mathrm{SL}(2,\R)$, let $s(D)\in\R\PP^1$ denotes the most contracted direction of $D$. It is a relatively straightforward fact that as a map from $\mathrm{SL}(2,\R)$ to $\R\PP^1$, $s(\cdot)$ is $C^\infty$ away from $\mathrm{SO}(2,\R)$ (see e.g. \cite[Lemma 10]{zhang}). Let $\vec s(D)\in s(D)$ denotes an unit vector. Similarly, let $u(D)=s(D^{-1})$ and $\vec u(D)\in u(D)$ be an unit vector. Then a standard polar decomposition procedure shows that
\beq\label{eq:polardecom}
D=R_{u(D)}\begin{pmatrix}\|D\|&0\\ 0& \|D\|^{-1}\end{pmatrix}R_{\frac\pi2-s(D)},
\eeq
where $R_\theta$ denotes the rotation matrix with rotation angle $\theta$:
$$
R_\t=\begin{pmatrix}
\cos\t, & -\sin\t\\ \sin\t, &\cos\t
\end{pmatrix}.
$$
 Let $s^\perp(D)$ denotes the orthogonal direction of $s(D)$. Then by \eqref{eq:polardecom}, it is clear that $s^\perp(D)$ is the most expanding direction of $D$. Similarly, let $\vec s^\perp(D)$ denote an unit vector in $s^\perp(D)$. Defining
\beq\label{eq:AsympDirections}
s_n(j)=s[A_n(j)] \mbox{ and } u_n(j)=s[A_{-n}(j)] 
\eeq
which we call \emph{$n$-step stable and unstable directions} or simply \emph{asymptotic stable and unstable directions}. As we will show in Lemmas~\ref{l.existenceus} and \ref{l.distanceus} below, they tend to the actual stable and unstable directions as $n$ goes to infinity. This pair of functions play key role throughout this paper. In the remaining part of this paper, let $|x|$ denotes $\min\{|x|, |x-\pi|\}$ for $x\in\R\PP^1=\R/(\pi\Z)$.

Let  $A(j):\Z\to\mathrm{SL}(2,\R)$ be bounded, i.e. $\|A\|_\infty<M$, and satisfying the uniform exponential growth condition~\eqref{eq:ueg}, i.e. $\|A_n(j)\|>c\l^n$ for all $j\in\Z$ and all $n
\ge 1$, where $\l>1$ and $c>0$ are independent of $n$ and $j$. We start with the existence of a pair of instinct invariant directions.

\begin{lemma}\label{l.existenceus}
There exist $u$ and $s:\Z\rightarrow\R\PP^1$ such that
\beq\label{eq:uniform_convergence_us_1}
\lim_{n\rightarrow\infty}\|u_n-u\|_\infty=\lim_{n\rightarrow\infty}\|s_n-s\|_\infty=0.
\eeq
Moreover, $u(j)\neq s(j)$, for all $j\in\Z$, and they are both $A$-invariant.
\end{lemma}

\begin{proof}
 By definition we have
\beq\label{eq:large_contraction1}
\|A_n(j)\vec s_{n+1}(j)\|=\|A(j+n)^{-1}A_{n+1}(j)\vec s_{n+1}(j)\|\le M\|A_{n+1}(j)\|^{-1}\le C\lambda^{-n}.
\eeq
Let $\t=|s_n(j)-s_{n+1}(j)|$. Then we may write $\vec s_{n+1}(j)=(\cos\t)\vec s_{n}(j)+(\sin\t)\vec s^\perp_{n}(j)$. By the definitions of $\vec s(D)$ and $\vec s^\perp(D)$and applying $A_n(j)$ at both sides, we obtain
\begin{align}
\label{eq:angle_control}	|\sin\t|\cdot\|A_n(j)\|&=|\sin\t|\cdot\|A_n(j)\vec s^\perp_{n}(j)\|\\
\nonumber	&=\|A_n(j)\vec s_{n+1}(j)-(\cos\t)A_n(j)\vec s_{n}(j)\|\\
\nonumber	&\le \|A_n(j)\vec s_{n+1}(j)\|+|\cos\t|\cdot\|A_n(j)\vec s_{n}(j)\|\\
\nonumber &\le C\l^{-n}+\|A_n(j)\|^{-1}.
	\end{align}
Since $\|A_n(j)\|^{-1}\le C\lambda^{-n}$, \eqref{eq:angle_control} then implies that
\beq\label{eq:exp_decay_error_s_n}
|s_n(j)-s_{n+1}(j)|=\t\le C|\sin\t|\le C\lambda^{-n}\|A_n(j)\|^{-1}\le C\lambda^{-2n}.
\eeq
Thus $\{s_n(j)\}_{n\in\Z}$ is a Cauchy sequence for each $j\in\Z$ and convergence is independent of $j\in\Z$. Thus there exists some $s:\Z\rightarrow\R\PP^1$ 
\beq\label{eq:approx_to_actual}
\|s_n-s\|_\infty\le C\lambda^{-2n}.
\eeq
In particular, $\lim_{n\rightarrow\infty}\|s_n-s\|_\infty=0$. Similarly, we get all the estimates for $u_n$ and $u$.

For the invariance property, we only need to note that
\beq\label{eq:large_contraction2}
\|A_n(j+1)A(j)\vec s_{n+1}(j)\|=\|A_{n+1}(j)\vec s_{n+1}(j)\|=\|A_{n+1}(j)\|^{-1}\le C\lambda^{-(n+1)}.
\eeq
Clearly,  \eqref{eq:large_contraction2} is similar to \eqref{eq:large_contraction1}. Thus, replacing $A_n(j)$ by $A_n(j+1)$, $s_{n+1}(j)$ by $A(j)\vec s_{n+1}(j)$, and $s_n(j)$ by $s_n(j+1)$, the same argument obtaining \eqref{eq:exp_decay_error_s_n} yields
$$
|A(j)\cdot s_{n+1}(j)-s_n(j+1)|\le C\lambda^{-2n},
$$
for all $n\ge1$ and all $j\in\Z$. So we have
$$
A(j)\cdot s(j)=A(j)\cdot[\lim_{n\rightarrow\infty}s_{n+1}(j)]=\lim_{n\rightarrow\infty}A(j)\cdot s_{n+1}(j)=\lim_{n\rightarrow\infty}s_n(j+1)=s(j+1).
$$
Similarly, we get that $u$ is also $A$-invariant.

To show $u(j)\neq s(j)$, by invariance property, we only need to show that $u(j_0)\neq s(j_0)$ for some $j_0$. First, we claim that there exists a pair $(c,\l)$ satisfying the uniform exponential growth condition \eqref{eq:ueg} and the following condition:
\beq\label{eq:upperbound}
\mbox{for all } N\in\Z+, \mbox{ there exist } j_0\in\Z \mbox{ and } n_0\ge N \mbox{ such that } \|A_{n_0}(j_0)\|\le c\l^{\frac32n_0}.
\eeq
First we clearly have that $\|A_n(j)\|\le M^n$ for all $j\in\Z$ and all $n\ge 1$. Now we start with a pair $(c_0,\l_0)$ satisfying \eqref{eq:ueg}. If \eqref{eq:upperbound} also holds true for $(c_0,\l_0)$, then we are done. Otherwise, \eqref{eq:upperbound} is false for $(c_0,\l_0)$. Then there exists a $N_1\in \Z_+$ such that for all $j\in\Z$ and all $n\ge N_1$, it holds that $\|A_n(j)\|\ge c_0\l_0^{\frac32n}$. 

We now set $\l_1=\l_0^{\frac32}$ and $c_1=\min\{c_0, \l_1^{-N_1}\}$. We claim that $A$ satisfies \eqref{eq:ueg} for the new pair $(c_1, \l_1)$. Indeed, for $n<N_1$, we have
$$
\|A_n(j)\|\ge 1\ge c_1\l_1^{N_1}\ge c_1\l_1^{n}
$$
and for $n\ge N_1$, it holds that
$$
\|A_n(j)\|\ge c_0\l_0^{\frac32n}=c_0\l^{n}_1\ge c_1\l_1^{n}.
$$
We repeat the process with the new pair $(c_1,\l_1)$. Since $\l_0>1$, the process must terminate before step $k$ where $\l_0^{(3/2)^k}>M$. Thus, we find the pair $(c,\l)$ with the properties \eqref{eq:ueg} and \eqref{eq:upperbound}.

We now work with the pair $(c,\l)$ as above. Let $j_0$ and $n_0\ge N$ be from \eqref{eq:upperbound} for some $N$. By \eqref{eq:approx_to_actual}, it holds that $|s_{n_0}(j_0)-s(j_0)|<C\l^{-2n_0}$. Write 
$$
\vec s(j_0)=\cos(s(j_0)-s_{n_0}(j_0))\vec s_{n_0}(j_0)+\sin(s(j_0)-s_{n_0}(j_0))\vec s^\perp_{n_0}(j_0).
$$ 
Combining everything together, we then have 
\begin{align*}
	\|A_{n_0}(j_0)\vec s(j_0)\|&=\|A_{n_0}(j_0)\cos(s(j_0)-s_{n_0}(j_0))\vec s_{n_0}(j_0)+A_{n_0}(j_0)\sin(s(j_0)-s_{n_0}(j_0))\vec s^\perp_{n_0}(j_0)\|\\
	&\le \|A_{n_0}(j_0)\vec s_{n_0}(j_0)\|+\|A_{n_0}(j_0)\vec s^\perp_{n_0}(j_0)\|\cdot|s(j_0)-s_{n_0}(j_0)|\\
	&\le \|A_{n_0}(j_0)\|^{-1}+\|A_{n_0}(j_0)\|C\l^{-2n_0}\\
	&\le C\l^{-{n_0}}+C\lambda^{\frac{3}2n_0}\l^{-2n_0}\\
	&\le C\l^{-n_0/2}.
\end{align*}
Thus for large $N$, it holds that $\|A_{n_0}(j_0)\vec s(j_0)\|<C\l^{-n_0/2}<1$. Similarly, we can get that 
$$
\|A_{-n_0}(j_0+n_0)\vec{u}(j_0+n_0)\|<1.
$$ 
Since $A_{n_0}(j_0)\cdot u(j_0)=u(j_0+n_0)$ and $A_{n_0}(j_0)^{-1}=A_{-n_0}(j_0+n_0)$, we then have
$$
\|A_{n_0}(j_0)\vec{u}(j_0)\|>1,
$$
which implies that $u(j_0)\neq s(j_0)$, concluding the proof.

\end{proof}
Next we show that $s$ and $u$ are away from each other with distances bounded uniformly from below.
\begin{lemma}\label{l.distanceus}
There exists a $\gamma>0$ in $\R\PP^1$ such that 
\beq\label{eq:lowerbound_su}
\inf_{j\in\Z}|s(j)-u(j)|\ge\gamma.
\eeq
\end{lemma}

\begin{proof}
Let $(\Omega, T, F)$ be as in Corollary~\ref{c:uniformeg2}. In other words, $\Omega=\mathrm{Hull}(A)$, $T$ is the left shift map, and $F(\omega)=\omega_0\in\SL(2,\R)$ which is clearly continuous. Then by uniform growth condition of $A$ and Corollary~\ref{c:uniformeg2}, it holds that
\beq\label{eq:ueg_stom}
\|F_n(\omega)\|\ge c\lambda^n, \mbox{ for all }\omega\in\Omega \mbox{ and all } n\ge1.
\eeq

 For each $\omega$, we define $B^\omega(j)=F(T^j\omega)$ for all $j\in\Z$. Let $s^\omega_n(j)$ and $u^\omega_n(j)$ be the asymptotic stable and unstable directions of $B^\omega$. By \eqref{eq:ueg_stom}, for each $\omega$, we could treat the sequence $\{B^\omega(j), j\in\Z\}$ as the sequence $\{A(j), j\in\Z\}$ in Lemma~\ref{l.existenceus}. If we define 
$$
 s_n(\omega)=s(F_n(\omega)) \mbox{ and } u_n(\omega)=s(F_{-n}(\omega)),
$$
 then $s_n(\omega)=s^\omega_n(0)$ and $u_n(\omega)=u^{\omega}_n(0)$. Applying \eqref{eq:exp_decay_error_s_n} to $B^\omega$ with $j=0$, we then obtain 
 $$
 |s_n(\omega)-s_{n+1}(\omega)|<C\l^{-2n}\mbox{ and }  |u_n(\omega)-u_{n+1}(\omega)|<C\l^{-2n},
 $$
 where from the proof of Lemma~\ref{l.existenceus}, $C$ has nothing to do with $\omega$. Hence, there exist some $u, s:\Omega\rightarrow\R\PP^1$, it holds that
\beq\label{eq:uniform_convergence_us2}
\lim_{n\rightarrow\infty}\|u_n-u\|_\infty=\lim_{n\rightarrow\infty}\|s_n-s\|_\infty=0.
\eeq
Note that $s(T^j\omega)$ is the stable direction $s^\omega(j)$ of $B^\omega$. In particular, for each $\omega$, it holds that 
$$
F(\omega)\cdot s(\omega)=B^\omega(0)\cdot s^\omega(0)=s^\omega(1)=s(T\omega),
$$ 
i.e. $s$ is $(T,F)$-invariant. Similarly, $u$ is $(T,F)$-invariant as well. Moreover, for each $\omega$, $s^\omega(0)\neq u^\omega(0)$ which implies that $s(\omega)\neq u(\omega)$ for each $\omega\in\Omega$. 

On the other hand, the most contracted direction is $C^\infty$ away from $\mathrm{SO}(2,\R)$, $u_n(\omega)$ and $s_n(\omega)$ are continuous in $\omega$ for all large $n$ by \eqref{eq:ueg_stom}. Thus, \eqref{eq:uniform_convergence_us_1} implies that $u$ and $s$ are continuous as well. Thus by compactness of $\Omega$ and continuity of $u$, $s$, we have for some $\gamma>0$
$$
|u(\omega)-s(\omega)|\ge\gamma,\ \mbox{for all }\omega\in\Omega.
$$
Since $A\in\Omega$ and $u(j)=u[T^j(A)]$, we get $|u(j)-s(j)|\ge\gamma$, for all $j\in\Z$.
\end{proof}

Finally, we need the following lemma, which is essentially \cite[Lemma~11]{zhang} adapted to the present setting. The proof is presented for completeness. 

\begin{lemma}\label{l:inv_cone}
	Let $\Omega$ be a set and $T:\Omega\to\Omega$ a bijection. Consider two maps $b:\Omega\to\R\setminus\{0\}$ and $\beta:\Omega\to\R\PP^1$. Define $D:\Omega\to\mathrm{SL}(2,\R)$ as
	$$
	D(\omega)=\begin{pmatrix}
b(\omega) &0\\ 0 &b(\omega)^{-1}
	\end{pmatrix}\cdot R_{\frac\pi2-\b(\omega)}
	$$
	 and consider the cocycle $(T,D)$. Define $\rho(\delta):=(\delta+\frac3{\delta})(\frac{\delta}2+1)$ for $\delta>0$. Assume there is a $\delta_0>0$ and a $\l_0>\rho(\delta_0)$ with the following properties:  for all $\omega\in\Omega$, it holds that
	 $$
	 |\tan\b(\omega)|>\delta_0 \mbox{ and }|b(\omega)|>\l_0.
	 $$ 
	 Then if we define $\CF\subset\R\PP^1$ to be $\CF:=\{\t\in\R\PP^1: |\tan\t|<\delta_0/2\}$, it holds that
	\begin{enumerate}
		\item $\CF$ is $(T,D)$-invariant, i.e. $D(\omega)\cdot \CF\subset \CF$.
		\item There is an $\a>1$ such that $\|D(\omega)\vec \t\|>\a$ for all unit vector $\vec \t$ with $\t\in\CF$.
		\end{enumerate} 
	\end{lemma}
\begin{proof}
	Let $\t\in\CF$. It is straightforward to see that the two variable function $g(t,r)=\frac{1+tr}{t-r}$, $t>r>0$ is decreasing in $t$ and increasing in $r$. Then a direction computation shows that
	\begin{align*}
	\left|\tan[D(\omega)\cdot\t]\right|&=\left|a^{-2}(\omega)\frac{1+\tan[\b(\omega)]\tan\t}{\tan[\b(\omega)]-\tan\t}\right|\\
	&\le \l_0^{-2}\frac{1+|\tan[\b(\omega)]|\cdot|\tan\t|}{|\tan[\b(\omega)]|-|\tan\t|}\\
	&\le \frac{2+\delta_0^2}{\l^2_0\delta_0}\\
	&<\delta_0/2,
	\end{align*}
	where the last inequality follows from the fact that $\l_0^2>(\delta_0+3/\delta_0)^2>4(1+\delta^2_0)/\delta^2_0.$ Note a similar estimate as above also shows that for all $\t\in\CF$,
		\begin{align*}
	\left|\tan\left[R_{\frac\pi2-\b(\omega)}\cdot\t\right]\right|&=\left|\frac{1+\tan[\b(\omega)]\tan\t}{\tan[\b(\omega)]-\tan\t}\right|\\
	&\le \frac{1+|\tan[\b(\omega)]|\cdot|\tan\t|}{|\tan[\b(\omega)]|-|\tan\t|}\\
	&<\frac{2+\delta_0^2}{\delta_0}.
	\end{align*}
	Now take $\t\in\CF$, then we may write $\vec \t=\binom{1}{r}/\sqrt{1+r^2}$ with $|r|<\delta_0/2$. Then the above estimate shows that we may write
	$$
	R_{\frac\pi2-\b(\omega)}\vec\t=\frac1{\sqrt{(1+t^2)(1+r^2)}}\binom{1}{t},
	$$
	where $|t|<\frac{2+\delta^2_0}{\delta_0}$. Thus we obtain
	\begin{align*}
	\|D(\omega)\vec \t\|^2&=\left\|\begin{pmatrix}
	b(\omega), &0\\ 0, &b(\omega)^{-1}
	\end{pmatrix}\cdot R_{\frac\pi2-\b(\omega)}\vec\t\right\|^2\\
	&=\frac{b^2(\omega)+b^{-2}(\omega)t^2}{(1+t^2)(1+r^2)}\\
	&\ge \frac{\l_0^2}{\left[1+(\frac{\delta_0}{2})^2\right]\left[1+(\frac2{\delta_0}+\delta_0)^2\right]}\\
	&>\frac{(\frac{\delta_0}2+1)^2}{1+(\frac{\delta_0}{2})^2}\cdot
	\frac{(\delta_0+\frac3{\delta_0})^2}{1+(\frac2{\delta_0}+\delta_0)^2}\\
	&>1,
	\end{align*}
	concluding the proof.
\end{proof}
Lemma~\ref{l:inv_cone} basically says that under its conditions, $(T,D)$ admits a constant invariant cone field. Moreover, for each vector in this cone field, the norm expands uniformly under the cocycle iteration. It clearly implies the uniform exponential growth condition. Note no addition structure of $\Omega$ is assumed in Lemma~\ref{l:inv_cone}.

Now we are ready to prove Theorem~\ref{t.uiformeg}. Since we already have invariance by Lemma~\ref{l.existenceus}, it suffices to show that vectors in the direction of $u$ are uniformly exponentially contracted backward under iteration of $A$ while vectors in the direction of $s$ are contracted forward. We only prove it for the $u$-direction since the proof of the $s$-direction is completely analogous.

\begin{proof}[Proof of Theorem~\ref{t.uiformeg}]
We show that there is some $\lambda_0>1$ such that for all $n\in\Z_+$ and all $j\in\Z$, it holds that
$$
\|A_{-n}(j)\vec{u}(j)\|\le C\lambda_0^{-n}.
$$
By $A$-invariance of $u$, we may equivalently show for all $n\in\Z_+$ and all $j\in\Z$ that
$$
\|A_n(j)\vec{u}(j)\|\ge c\lambda_0^{n}.
$$
By proofs of Lemma~\ref{l.existenceus} and Lemma~\ref{l.distanceus}, there exists an $N\in\Z_+$ such that for all $k\ge N$ and for all $j\in\Z$, it holds that
\begin{align*}
|u_k(j)-s_k(j)|&>\frac\gamma2,\\  
|u(j)-u_k(j)|&<C\lambda^{-2k}:=\varepsilon_k, \mbox{ and }\\
\|A_k(j)\|&\ge c\lambda^k:=\Lambda_k.
\end{align*}
 By definition, we have
$$
A_k(j)=R_{u_k(j+k)}\begin{pmatrix}\|A_k(j)\|&0\\0&\|A_k(j)\|^{-1}\end{pmatrix}R_{\frac\pi2-s_k(j)}.
$$
For each $k\ge N$, we may consider a new map $B^{(k)}:\Z\rightarrow\mathrm{SL}(2,\R)$ such that
\beq\label{eq:conj1}
B^{(k)}(j):=\begin{pmatrix}\|A_k(jk)\|&0\\0&\|A_k(jk)\|^{-1}\end{pmatrix}R_{\frac\pi2+u_k(jk)-s_k(jk)}.
\eeq
Then it is clear that 
\beq\label{eq:conjugate}
B^{(k)}(j)=R_{-u_k((j+1)k)}A_k(jk)R_{u_k(j)}.
\eeq 
We may choose $k$ large so that $\Lambda_k>\rho(\frac\gamma2)$ and $|\tan\e_k|<\frac\gamma4$. Fix such a $k$ and apply Lemma~\ref{l:inv_cone} to $(\Z,T,B^{(k)})$ where $T(j)=j+1$. For each $j$, let $\vec v_j\in [u(jk)-u_k(jk)]$ be an unit vector. We then obtain that for some $\a>1$ it holds true for each $j\in\Z$ that 
$$
\|B^{(k)}_n(j)\vec v_j\|\ge \alpha^n,\ \mbox{for all }j\in\Z \mbox{ and all }n\in\Z^+.
$$

Then we may pass the estimate to $A$ via \eqref{eq:conj1} as follows. For each pair $(j,n)$, we may first find some $q$ so that $(q-1)k< j\le qk$. Then we may write $n-(qk-j)=kr+p$ where $0\le p<k$. Note $r\ge [\frac{n}{k}]-2$. Then by \eqref{eq:conjugate}, we obtain for all $j\in\Z$ and all $n\in\Z^+$ that
\begin{align}
\nonumber \|A_n(j)\vec{u}(j)\|&=\|A_{p}(kr+qk)\cdot A_{kr}(qk)\cdot A_{qk-j}(j)\vec u(j)\|\\
\nonumber &\ge c\|B^{(k)}_r(q)\cdot R_{-u_k(qk)}\cdot A_{qk-j}(j)\vec u(j)\|\\
\nonumber &\ge c\|B^{(k)}_r(q)\vec v_{q}\|\\
\nonumber &\ge c\alpha^{r}\\
\label{eq:exp-norm-growth}&\ge c[\alpha^{\frac1k}]^n,
\end{align}
where the first inequality we use the fact that $\|A_p(kr+qk)\|^{-1}>c$ for some $c=c(M,k)$, equation \eqref{eq:conjugate}, and the fact that rotation matrices preserve the operator norm; for the second the inequality, we use the fact that $0<qk-j<k$ which implies $R_{-u_k(qk)}\cdot A_{qk-j}(j)\vec u(j)=cR_{-u_k(qk)}\vec u(qk)$ for some $c=c(M,k)$. Evidently, $R_{-u_k(qk)}\vec u(qk)$ is an unit vector in the direction of $u(qk)-u_k(qk)$ which is by definition is $\vec v_q$. 

Clearly, \eqref{eq:exp-norm-growth} is the desired estimate with $\lambda_0=\alpha^{\frac1k}$. Similarly, we may get the estimate for $s$. Thus, as the notation suggested, we show that $u$ is the unstable direction of $A$ as in Definition~\ref{d.uhsequence} and $s$ is the stable direction. This completes the proof of the Theorem~\ref{t.uiformeg}.
\end{proof}
\begin{remark}\label{r:uniformeg}
	The proof of Corollary~\ref{c:uniformeg} is identical to the proof of Theorem~\ref{t.uiformeg}. One only needs to replace $\Z$ by $\Omega$, $A:\Z\to\SL(2,\R)$ by $A:\Omega\to\SL(2,\R)$, and $A_n(j)$ by $A_n(\omega)$. In fact, the proof of Corollary~\ref{c:uniformeg} is even simpler: there is no need to introduce the \emph{Hull} of a sequence. 
\end{remark}
From the proof Lemma~\ref{l.distanceus}, the proof of Theorem~\ref{t.uiformeg}, Lemma \ref{l:inv_cone}, and Corollary~\ref{c:uniformeg}, it is actually not difficult to deduce the following Corollary~\ref{c:SeparationUStoUH} which give another equivalent condition of uniform hyperbolicity. Let $(\Omega,T)$ and $A:\Omega\to\SL(2,\R)$ be as in Corollary~\ref{c:uniformeg} and consider cocycle dynamics $(T,A)$ as \ref{eq:cocycle0}. For $n\ge 1$, let $s_n(\omega)=s[A_n(\omega)]$ and $u_n(\omega)=s[A_{-n}(\omega)]$.
\begin{corollary}\label{c:SeparationUStoUH}
$(T,A)\in\CU\CH$ if and only if there are some $k\in\Z_+$ and some $\delta_0>0$ so that the following hold ture for all $\omega\in\Omega$:
\begin{align}
\label{eq:largenorm}& \|A_k(\omega)\|> \rho(\delta_0), \\
\label{eq:largedistance}& |s_k(\omega)-u_k(\omega)|>\delta_0.
 \end{align}
\end{corollary}
\begin{proof}
The only if part is quite straightforward. $\CU\CH$ implies the uniform exponential growth of $\|A_n(\omega)\|$. Hence, by the proof of Lemma~\ref{l.distanceus}, we obtain $|u(\omega)-s(\omega)|>\gamma$ for all $\omega\in\Omega$ and for some $\gamma>0$, and the uniform convergence of $s_n(\omega)$ (resp. $u_n(\omega)$) to $s(\omega)$ (resp. $u(\omega)$). This clearly implies for all large $k$, \eqref{eq:largedistance} holds true with $\delta_0=\frac\gamma2$. By the uniform exponential growth condition, we may pick some $k$ large so that \eqref{eq:largenorm} holds true.

For the if part, similar to the proof of Theorem~\ref{t.uiformeg}, conditions~\eqref{eq:largenorm} and ~\eqref{eq:largedistance} imply the existence of a constant invariant cone field $\CF:\Omega\to\R\PP^1$ for the new cocycle $(T^k,B)$ where
$$
B(\omega):=\begin{pmatrix}\|A_k(\omega)\|&0\\0&\|A_k(\omega)\|^{-1}\end{pmatrix}
R_{\frac\pi2+u_k(\omega)-s_k(\omega)}.
$$
It holds that $\|B_n(\omega)\|=\|A_{nk}(\omega)\|$ for all $n\in\Z_+$ since
$$
B(\omega)=R_{-u_k(T^k\omega)}A_k(\omega)R_{u_k(\omega)}.
$$
  Then apply Lemma~\ref{l:inv_cone} to $(T^k,B)$, we obtain for some $\a>1$ that $\|B(\omega)\vec v\|>\a$ for all $\omega\in\Omega$ and for all unit vector $\vec v\in \CF$. It clearly implies the uniform exponential growth of $(T,B)$, hence the uniform exponential growth of $(T,A)$. By Corollary~\ref{c:uniformeg}, one then gets $(T,A)\in\CU\CH$.
\end{proof}

The main advantage of this description is that while other definitions involve $n$-step cocycle iterations for all $n\in\Z_+$, Corollary~\ref{c:SeparationUStoUH} only need information for cocycle iterations up to a certain finite step $k$. Hence, it may become a useful tool to produce uniformly hyperbolic systems. In fact, this is exactly one of the key ideas of \cite{wangzhang1} to show \emph{Cantor Spectrum} (i.e. the spectrum is a Cantor set) for a class of quasiperiodic Schr\"odinger operators. One may have a more enhanced version of Corollary~\ref{c:SeparationUStoUH}, see e.g. \cite[Lemma 5]{zhang2}. This idea of separation of asymptotic stable and unstable directions are promising in the sense that it may further be used to get more results concerning Cantor spectrum, which is another central topic in the spectral analysis of quasiperiodic Schr\"odinger operators.  

To show Corollary~\ref{c:SeparationUStoUH}'s usefulness, we give the following almost immediate consequence of Corollary~\ref{c:SeparationUStoUH}. Let $s^{(\cdot)}_n, s^{(\cdot)}:\Omega\to\R\PP^1$ denote the $n$-step asymptotic stable and the stable directions of $(T,\cdot)\in\CU\CH$. Similarly, we can define $u^{(\cdot)}_n$ and $u^{(\cdot)}$. Recall that the most contracted direction map is $C^\infty$ away from $\mathrm{SO}(2,\R)$. Thus by compactness of $\Omega$, $s^{(\cdot)}_n, u^{(\cdot)}_n: C^0(\Omega, \mathrm{SL}(2,\R))\to C^0(\Omega,\R\PP^1)$ are continuous as long as the $n$-step cocycle iterations are uniformly away from $\mathrm{SO}(2,\R)$, i.e. for some $\a>1$, $\|(\cdot)_n(\omega)\|>\a$ for all $\omega\in\Omega$.

\begin{corollary}\label{c:cont_su}
	Let $(\Omega,T,A)$ be as in Corollary~\ref{c:uniformeg}. Suppose $(T,A)\in\CU\CH$. Then for any $\e>0$, there exists a $\delta>0$ so that if $B:\Omega\to\mathrm{SL}(2,\R)$ satisfies $\|A-B\|_{\infty}<\delta$, then $(T,B)\in\CU\CH$. Moreover, it holds that
	\beq\label{eq:cont_su}
	\|s^{(A)}-s^{(B)}\|_{\infty}<\e \mbox{ and }\|u^{(A)}-u^{(B)}\|_{\infty}<\e.
	\eeq
	In other words, $s^{(\cdot)}, u^{(\cdot)}:C^0\left(\Omega, \mathrm{SL}(2,\R)\right)\cap\CU\CH\to C^0(\Omega,\R\PP^1)$ are continuous.
\end{corollary}
\begin{proof}
By Corollary~\ref{c:SeparationUStoUH}, $(T,A)\in\CU\CH$ implies that \eqref{eq:largenorm}-\eqref{eq:largedistance} hold true for some $k\in\Z_+$ and $\l>1$. Then by the fact stated above Corollary~\ref{c:cont_su}, there is a $\delta>0$ such that if $\|B-A\|_{\infty}<\delta$, then \eqref{eq:largenorm}-\eqref{eq:largedistance} hold true for $(T,B)$ by the fact stated above Corollary~\ref{c:cont_su}. Hence, $(T,B)\in\CU\CH$ by Corollary~\ref{c:SeparationUStoUH}. 

For the proof of \eqref{eq:cont_su}, by the proof of \eqref{eq:uniform_convergence_us_1} in Lemma~\ref{l.existenceus}, $s_n$ (resp. $u_n$) converges to $s$ (resp. $u$) uniformly in $\omega\in\Omega$. So we may fix a $N\in\Z_+$ large so that
	$$
	\|s^{(\star)}-s^{(\star)}_N\|_\infty< \e/3 \mbox{ and } \|u^{(\star)}-u^{(\star)}_N\|_\infty<\e/3,
	$$
	where $\star=A$ or $B$, and so that $\|A_N(\omega)\|>\a$ and $\|B_N(\omega)\|>\a$ for some $\a>1$ and for all $\omega\in\Omega$.  Choosing $\delta$ small, by the fact stated above the Corollary~\ref{c:cont_su}, it holds that
	$$
	\|s^{(A)}_N-s^{(B)}_N\|_\infty<\e/3 \mbox{ and } \|u^{(A)}_N-u^{(B)}_N\|_\infty<\e/3.
	$$
	Hence, \eqref{eq:cont_su} is a consequence of the triangle inequality.
	\end{proof}

\subsection{No Nontrivial Bounded Orbit Implies Uniform Hyperbolicty} The main part of the proof of Theorem~\ref{l.nontrivialbo} is the only if part, i.e. we need to show if not $\CU\CH$, then \eqref{eq:NonTrivailBounded} holds. To this end, we need to argue by contradiction and show the contrary implies $\CU\CH$. Recall here $\Omega$ is a compact metrice space.
\begin{proof}[Proof of Theorem~\ref{l.nontrivialbo}]
	For the proof of if part, we note that $\CU\CH$ implies that for all $\vec v\in\R^2$ and all $\omega\in\Omega$, it holds that
	$$
	\vec v=c_1\vec u(\omega)+c_2\vec s(\omega), \mbox{ for some } c_1,\ c_2\in\R.
	$$
	Thus if $\vec v\neq0$, then we must have that $\|A_n(\omega)v\|$ grows exponentially fast either as $n\rightarrow\infty$ or as $n\rightarrow-\infty$.
	
	For the proof of only if part, the main thing is to show the following.  Suppose there exist $\varepsilon>0$ and $L\in\Z^+$ with the following property: for all $(\omega,\vec v)\in\Omega\times\mathbb S^1$, there is a $|l|\le L$ such that
	$$
	\|A_l(\omega)\vec v\|\ge1+\varepsilon.
	$$
	Then we claim $(T,A)$ satisfies uniform exponential growth condition. 
	
	For $(\omega,\vec v)\in\Omega\times\mathbb S^1$, let $l(\omega,\vec v)$ be defined as
	\beq\label{eq:FirstLargeTime}
	|l(\omega,\vec v)|=\min\{|l|: |l|\le L \mbox{ and } \|A_l(\omega)\vec v\|\ge1+\varepsilon\}
	\eeq
For each $(\omega,\vec v)\in\Omega\times\SS^1$, we then define the following sequence $(l_k, \vec v_k, \omega_k)_{k\ge 0}$ by induction:
	$$
	l_0=0,\ \vec v_0=\vec v,\mbox{ and } \omega_0=\omega,
	$$
and for each $k\ge1$,
	$$
	l_k=l(\omega_{k-1},\vec v_{k-1}),\ \vec v_k=\frac{A_{l_{k}}(\omega_{k-1})\vec v_{k-1}}{\|A_{l_{k}}(\omega_{k-1})\vec v_{k-1}\|},\mbox{ and }\omega_k=T^{l_{k}}(\omega_{k-1}).
	$$
Then it is straightforward to see that for each pair $(p,k)$ with $0\le p\le k-1$, it holds that
\begin{align}
	\nonumber \|A_{l_k+l_{k-1}+\cdot+l_{p+1}}(\omega_{p})\vec v_{p}\|&=\|A_{l_{k-1}+\cdot+l_{p+1}}(\omega_p)\vec v_{p}\|\cdot \|A_{l_k}(\omega_{k-1})\vec v_{k-1}\|\\
\nonumber 	&=\prod^{k-1}_{j=p}\|A_{l_{j+1}}(\omega_j)\vec v_j\|\\
\nonumber 	&\ge (1+\e)^{k-p}\\
\label{eq:ChainNormLarge}	&\ge 1+\e.
	\end{align}
	
For $p\in\Z$ and $q\in\Z_+$, define
 $$
 I_{q}(p)=[p-q+1, p+q-1]\subset\Z.
 $$ 
 Let $L_k=\sum^{k}_{j=0}l_j$. By the minimality of $l_k$ from \eqref{eq:FirstLargeTime}, we have for each $k\ge 2$,
	\beq\label{eq:ExtendedInterval1}
	L_k\notin I_{k-1}:=\bigcup^{k-1}_{p=1}I_{|l_p|}\left(L_{p-1}\right).
	\eeq
Indeed, if $L_k\in I_{k-1}$, then $L_k\in I_{|l_p|}(L_{p-1})$ for some $p\ge 1$, then $|L_k-L_{p-1}|<|l_{p}|$ and by \eqref{eq:ChainNormLarge}
$$
\|A_{L_k-L_{p-1}}(\omega_{p-1})\vec v_{p-1}\|=\|A_{l_k+\cdots+l_{p}}(\omega_{p-1})\vec v_{p-1}\|\ge 1+\e,
$$
which contradicts the minimality property of $l_{p}$.

Clearly, $L_k-L_{k-1}=l_{k}$ implies that $L_k$ is on the boundary of $I_{|l_k|}(L_{k-1})$. Hence by \eqref{eq:ExtendedInterval1}, $L_k$ is on the boundary of $I_{k-1}$. This in turn implies that $I_k$ is a connected interval in $\Z$. Moreover, by definition we have $l_j\neq 0$ for all $j>0$. Thus it must hold for all $k\ge0$ that
	\beq\label{eq:ExtendedInterval2}
	|I_{k+1}|\ge|I_k|+1,
	\eeq
which in particular implies that there exists a $0\le K\le L$ such that $|I_K|\ge L$ for all $(\omega,\vec v)$. Hence $|I_k|\ge L$ for all $k\ge K$. 

Next, since $|l_k|\le L$ for all $k\ge0$, by \eqref{eq:ExtendedInterval1} and the fact that $I_{k}$ is connected in $\Z$, we must have that $L_k$ is at the same side of $I_{k-1}$ for all $k>K$. In other words, as $k$ is getting large, the interval $I_k$ expands along the same direction for all $k\ge K$. Note $0=l_0\in I_k$ for all $k\ge 1$. Thus for each $(\omega,\vec v)$, we obtain for some $1\le K\le L$,
$$
\mbox{ either }L_{k+1}>L_{k}>0 \mbox{ for all } k\ge K; \mbox{ or } 0>L_{k}>L_{k+1} \mbox{ for all } k\ge K.
$$ 

Then we claim that for each $\omega\in\Omega$, there must exist a $\vec v_0\in\SS^1$ so that 
$$
L_{k+1}(\omega,\vec v_0)>L_{k}(\omega,\vec v_0)>0 \mbox{ for all }k\ge K.
$$
Indeed, suppose this is not ture. Then for some $\omega\in\Omega$ and for all $\vec v\in\SS^1$, it holds that 
$$
0>L_{k}>L_{k+1} \mbox{ for all } k\ge K.
$$
Note that $|L_k-L_{k-1}|=|l_k|\le L$. Hence $|L_K|\le KL\le L^2$. Thus for all $n\in\Z_{-}$ with $|n|>L^2$, there must exist some $k\ge K$ so that $0>L_k>n\ge L_{k+1}$. It clearly holds that 
$$
0< L_k-n<L\mbox{ and } k\ge\frac{|n|}{L}.
$$ 
Hence a similar argument as in \eqref{eq:ChainNormLarge} shows that:
\begin{align}
\label{eq:ueg1}\|A_n(\omega)\vec v\|&\ge c\|A_{L_{k}}(\omega)\vec v\|\\
\nonumber &=\|A_{l_k+l_{k-1}+\cdot+l_{1}}(\omega)\vec v\|\\
\nonumber &=\prod^{k-1}_{j=0}\|A_{l_{j+1}}(\omega_j)\vec v_j\|\\
\nonumber &\ge c(1+\e)^k\\
\nonumber &\ge c(1+\e)^{\frac{|n|}{L}},
\end{align}
where the estimates hold uniformly true for all $\omega\in\Omega$ and for all $\vec v\in\SS^1$. Choosing $n$ large and picking an unit vector $\vec v\in s[A_n(\omega)]$, we then obtain
$$
1>\|A_n(\omega)\|^{-1}=\|A_n(\omega)\vec v\|\ge c(1+\e)^{\frac{|n|}{L}}> 1,
$$
a contradiction. Consequently, for each $\omega$, we may choose some $\vec v_0$ so that 
$$
L_{k+1}(\omega,\vec v_0)>L_{k}(\omega,\vec v_0)>0\mbox{ for all }k\ge K.
$$ 
Then by a similar argument as the estimate \eqref{eq:ueg1}, we obtain for all $n> L^2$:
\beq\label{eq:ueg0}
\|A_n(\omega)\|\ge\|A_n(\omega)\vec v_0\|\ge c(1+\e)^\frac{n}{L}.
\eeq
Note $c$ is independent of $\omega$. Changing $c$ in \eqref{eq:ueg0} if necessary to incorporate all $n$ with $1\le n\le L^2$, we then obtain uniform exponential growth property of $(T,A)$. By Theorem~\ref{t.uiformeg}, $(T,A)\in\CU\CH$. 

Hence if we assume $(T,A)\notin\CU\CH$, then for all $\varepsilon>0$ and for all $L>0$, there exists some $(\omega,\vec v)\in\Omega\times\mathbb S^1$ such that for all $|l|\le L$ we have
	$$
	\|A_l(\omega)\vec v\|<1+\varepsilon.
	$$
	Thus for each $m\in\Z^+$, we get a $(\omega^{(m)}, \vec v^{(m)})\in\Omega\times\mathbb S^1$ satisfies the above condition with $\e=\frac1m$ and $L=m$. Note that $\Omega\times\mathbb S^1$ is a compact metric space since $\Omega$ is a compact metric space. Thus, by passing to a subsequence, we may assume for some $(\omega,\vec v)\in\Omega\times\mathbb S^1$ that
	$$
	\lim_{m\rightarrow\infty}(\omega^{(m)},\vec v^{(m)})=(\omega,\vec v)
	$$
	Thus we have for each $n\in\Z$, 
	$$
	\|A_n(\omega)\vec v\|\le \lim_{m\to\infty}\|A_n(\omega^{(m)})\vec v^{(m)}\|\le \lim_{m\to\infty}\left(1+\frac1m\right)=1,
	$$
	concluding the proof.
\end{proof}

\bigskip

\section{Johnson's Theorem for sequence potentials}\label{s:johnson}

In this Section, we prove Theorem~\ref{t.uniformhers}. Then we may deduce Theorem~\ref{t.uniformhers2}. For $\psi=(\psi_n)_{n\in\Z}\in\ell^2(\Z)$, let $\|\psi\|$ denotes $\ell^2$ norm of $\psi$, i.e. $\|\psi\|^2=\sum_{n\in\Z}|\psi_n|^2$.

\subsection{Uniform Hyperbolicity Implies Invertibility of the Operator}\label{ss:uh_to_resolvent}
Let us first show $\sigma(H_v)\subset\{E:A^{(E-v)}\notin\CU\CH\}$. Equivalently, we show
$$
\{E:A^{(E-v)}\in\CU\CH\}\subset\rho(H_v).
$$
Fix $E$ such that $A^{(E-v)}\in\CU\CH$. Then by Definition~\ref{d.uhsequence}, $A^{(E-v)}$ has unstable direction $u$ and stable direction $s$. Let $\psi^u, \psi^s\in\R^\Z$ be solution to the eigenfunction equation $H_v\psi=E\psi$ and be such that
$$
\binom{\psi^u_0}{\psi^u_{-1}}\in u(0), \binom{\psi^s_0}{\psi^s_{-1}}\in s(0).
$$
We normalize them so that
$$
\det\begin{pmatrix}\psi^s_0 &\psi^u_0\\\psi^s_{-1} &\psi^u_{-1}\end{pmatrix}=1.
$$
Thus we have for all $n\in\Z$,
\beq\label{eq:determinant1}
\det\begin{pmatrix}\psi^s_{n} &\psi^u_{n}\\\psi^s_{n-1} &\psi^u_{n-1}\end{pmatrix}=\det \left[A^{(E-v)}_n(0)\begin{pmatrix}\psi^s_{0} &\psi^u_{-1}\\\psi^s_{-1} &\psi^u_{-1}\end{pmatrix}\right]=1.
\eeq
Then we may construct the so-called Green's function $G:\Z^2\rightarrow\R$ of $H_v-E$ as:
\beq\label{eq:green}
G(p,q)=\begin{cases}\psi^u_{p}\cdot \psi^s_{q} & {\rm if}\ p\le q,\\\psi^u_{q}\cdot\psi^s_{p} & {\rm if}\ q<p.\end{cases}
\eeq
Note that $G(p,q)$ is symmetric, i.e. $G(p,q)=G(q,p)$. Below, we may often flip the $(p,q)$ in $G$.
\begin{lemma}\label{l:exponential_decay_of_Green's_Function}
	There  exist $C>0,\ \lambda>1$, independent of $(p,q)$, such that
	$$
	|G(p,q)|\le \frac C\gamma\lambda^{-|p-q|}, \mbox{ for all } (p,q)\in\Z^2.
	$$
Here $\gamma$ is from \eqref{eq:lowerbound_su}, i.e. the uniform lower bound between $u$ and $s$ in Lemma~\ref{l.distanceus}.
	\end{lemma}
\begin{proof}
	Let $\vec \psi^s(n)=\binom{\psi^s_{n}}{\psi^s_{n-1}}$ and $\vec \psi^u(n)=\binom{\psi^u_{n}}{\psi^u_{n-1}}$. Note also
$$
1=\det[\vec \psi^s(n), \vec \psi^u(n)]=\|\vec \psi^s(n)\|\cdot \|\vec \psi^u(n)\|\cdot |\sin(u(n)-s(n))|.
$$
Thus for all $n\in\Z$, it holds that
$$
\|\vec \psi^s(n)\|\cdot \|\vec \psi^u(n)\|\le \frac{C}\gamma.
$$
Without loss of generality, we assume $p\le q$. Then
\begin{align*}
	|G(p,q)|&=|\psi^u_p\psi^s_q|\\
	&\le \|\vec \psi^u(p)\|\cdot \|\vec \psi^s(p)\|\cdot \frac{\|\vec \psi^s(q)\|}{\|\vec \psi^s(p)\|}\\
	&\le\frac{C}{\gamma}\frac{\|A^{(E-v)}_{q-p}(p)\vec \psi^s(p)\|}{\|\vec \psi^s(p)\|}\\
	&\le \frac{C}{\gamma}\l^{-(q-p)},
\end{align*}
where the last inequality follows from $\vec \psi^s(n)\in s(n)$ for all $n\in\Z$. This concludes the proof.
\end{proof}
 Define the operator $S:\ell^2(\Z)\rightarrow\ell^2(\Z)$ so that
$$
(S\psi)_n=\sum_{p\in\Z}G(n,p)\psi_p,\ \psi\in\ell^2(\Z).
$$
First, we show that $(H_v-E)\circ S=Id$, i.e. $S$ is the inverse of $H_v-E$. For each $p\in\Z$, consider the squence $(G(p,n))_{n\in\Z}$. By the construction \eqref{eq:green} of $G$, it holds true that

\begin{align*}
[(H_v-E)G(p,\cdot)](n)&=G(p,n+1)+G(p,n-1)+(v(n)-E)G(p,n)\\
&=
\begin{cases}
\psi^u_p[\psi^{s}_{n+1}+\psi^{s}_{n-1}+(v(n)-E)\psi^s_n], &\ n\ge p+1\\
\psi^s_p[\psi^{u}_{n+1}+\psi^{u}_{n-1}+(v(n)-E)\psi^u_n], &\ n\le p-1.\\
\psi^s_{p+1}\psi^u_p+\psi^s_p[\psi^u_{p-1}+(v(p)-E)\psi^u_p], &\ n=p,
\end{cases}\\
&=
\begin{cases}
	0, &\ n\neq p\\
	\psi^s_{p+1}\psi^u_p-\psi^s_p\psi^u_{p+1}, &\ n=p,
	\end{cases}\\
&=\delta_{pn},
\end{align*}
where $\delta_{pn}$ is the standard notation which is $0$ if $n\neq p$ and is $1$ if $n=p$. Note in the last equality, we use \eqref{eq:determinant1}. Thus, for any $\phi=(\phi_n)\in\ell^2(\Z)$, it holds that
\begin{align*}
[(H_v-E)\circ S(\phi)]_n&=(S\phi)_{n+1}+(S\phi)_{n-1}+[v(n)-E](S\phi)_n\\
&=\sum_{p\in\Z}\left[G(p,n+1)+G(p,n-1)+(v(n)-E)G(p,n)\right]\phi_p\\
&=\sum_{p\in\Z}\delta_{pn}\phi_n\\
&=\phi_n,
\end{align*}
as desired.
 Next we show that $S$ is bounded. Let $\|G(n,\cdot)\|_{\ell^1}$ be the $\ell^1$ norm of the sequence $(G(n,p))_{p\in\Z}$. By Lemma~\ref{l:exponential_decay_of_Green's_Function}, $\|G(n,\cdot)\|_{\ell^1}<\frac{C}{\gamma}$ for all $n\in\Z$. Note the upper bound is independent of $n$. Then for all $\psi\in\ell^2(\Z)$, it holds that
\begin{align*}
\|S(\psi)\|^2&=\sum_{n\in\Z}|(S\psi)_n|^2=\sum_{n\in\Z}\left|\sum_{p\in\Z}G(n,p)\psi_p\right|^2\\
&\le \sum_{n\in\Z}\left(\sum_{p\in\Z}|G(n,p)|^\frac12|G(n,p)|^\frac12|\psi_p|\right)^2\\
&\le \sum_{n\in\Z}\left(\sum_{p\in\Z}|G(n,p)|\right)\left(\sum_{p\in\Z}|G(n,p)||\psi_p|^2\right)\\
&\le \frac{C}{\gamma}\sum_{n\in\Z}\sum_{p\in\Z}|G(n,p)||\psi_p|^2\\
&= \frac{C}{\gamma}\sum_{p\in\Z}\sum_{n\in\Z}|G(n,p)||\psi_p|^2\\
&= \frac{C}{\gamma}\sum_{p\in\Z}|\psi_p|^2\sum_{n\in\Z}|G(n,p)|\\
&\le \left(\frac{C}{\gamma}\right)^2\sum_{p\in\Z}|\psi_p|^2\\
&=\left(\frac{C}{\gamma}\right)^2\|\psi\|^2.
\end{align*}
Note we use Cauchy-Schwarz's inequality for the second inequality and Fubini's Theorem for the third equality. 
Hence, $H_v-E$ is invertible with the bounded inverse $S$, which implies that $E\in\rho(H_v)$. Here $G$ is the so called Green's function for $H_v-E$. 

Note that the estimate above shows that norm of the operator $S$ is related to the constant $\gamma$ via $\gamma<\frac{C}{\|S\|}$. Thus, as $E$ gets close to the spectrum, $\|S\|$ tends to $\infty$, the stable and unstable directions will tend to each other (at least somewhere).

\subsection{Uniform Hyperbolicty Away From The Spectrum}\label{ss:NotUHtoSpectrum} Now we show the other direction.  We will provide two different proofs. In Section~\ref{sss:resolvent_to_uh_1}, we show i.e. $\{E:A^{(E-v)}\notin\CU\CH\}\subset\sigma(H_v)$ via Theorem~\ref{l.nontrivialbo}. In Section~\ref{sss:resolvent_to_uh_2}, equivalently we show $\rho(H_v)\subset \{E:A^{(E-v)}\in\CU\CH\}$ via Combes-Thomas type of estimates.

\subsubsection{Non Uniform Hyperbolicty Implies Spectrum}\label{sss:resolvent_to_uh_1} We first have the following simple lemma. We omit the proof as it is an easy consequence of the Weyl's Criterion (see, for example, \cite{reedsimon}).

\begin{lemma}\label{l.finitesae}
	$E\in\sigma(H_v)$ if and only if for each $\varepsilon>0$, there exists a finitely supported unit vector $\psi\in\ell^2(\Z)$ such that
	$$
	\|(H_v-E)\psi\|<\varepsilon.
	$$
\end{lemma}

Then we need to embed $v$ to its Hull. Let $\Omega=\mathrm{Hull}(v)$, which in this case is clearly a compact metric space. Let $(\Omega, T, f)$ as defined Section~\ref{ss:johnson} and consider the Schr\"odinger cocycle $(T,A^{(E-f)})$. Clearly, $A^{(E-v)}\notin\CU\CH$ implies that $(T,A^{(E-f)})\notin\CU\CH$ since the former is a single orbit of the latter. By Theorem~\ref{l.nontrivialbo}, there is a $(\omega,\vec v)\in\Omega\times\mathbb S^1$ such that
$$
\|A^{(E-f)}_n(\omega)\vec v\|\le1,\ \mbox{for all } n\in\Z.
$$
Define $\psi\in\R^\Z$ such that
$$
\binom{\psi_n}{\psi_{n-1}}=A^{(E-f)}_n(\omega)\vec v,\ \mbox{for all } n\in\Z.
$$
Then it holds that $\|\psi\|_\infty\le 1$ and
$$
H_\omega \psi=E\psi.
$$
We claim that $E\in\sigma(H_\omega)$. Indeed, if $\|\psi\|\le C$, then $E$ is an eigenvalue of $H_\omega$. Hence, $E\in\sigma(H_v)$. Otherwise, if we define $\hat \psi^L$ as
$$
\hat \psi^L_n=\begin{cases}\psi_n, & {\rm if}\ |n|\le L,\\ 0, &{\rm otherwise},\end{cases}
$$
then $\|\hat \psi^L\|\rightarrow\infty$ as $L\rightarrow\infty$ and
$$
[(H_\omega-E)\hat \psi^L]_n=\begin{cases}\pm \psi_n, & {\rm if}\ n=\pm L, \pm (L+1)\\ 0, &{\rm otherwise}.\end{cases}
$$
Thus if we define $\psi^L=\frac{\hat \psi^L}{\|\hat \psi^L\|}$, then it holds that
$$
\|(H_\omega-E)\psi^L\|\le\frac{C}{\|\hat \psi^L\|},
$$
which can be arbitrary small as $L\rightarrow\infty$. By Lemma~\ref{l.finitesae}, $E\in\sigma(H_\omega)$. 

In both cases, by Lemma~\ref{l.finitesae}, for all $\varepsilon>0$, there exists a finitely supported unit vector $\psi\in\ell^2(\Z)$ such that $\|(H_\omega-E)\psi\|<\varepsilon$. Then $\omega\in\mathrm{Hull}(v)$ implies that there exists a $\{N_l\}_{l\in\Z}$ such that $T^{N_l}(v)$ converges to $\omega$ in the product topology. Since $\psi$ is finitely supported, we may choose $l$ large so that
$$
\|(H_{T^{N_l}(v)}-E)\psi\|<\varepsilon.
$$
Equivalently, we have
$$
\|(H_v-E)[T^{-N_l}(\psi)]\|<\varepsilon,
$$
where $(T\psi)_n=\psi_{n+1}$ is an unitary operator on $\ell^2(\Z)$. Hence $T^{-N_l}(\psi)$ is again finitely supported with norm $1$ which implies that $E\in\sigma(H_v)$ by Lemma~\ref{l.finitesae}.

\subsubsection{Uniform hyperbolicity via Combes-Thomas Estimate}\label{sss:resolvent_to_uh_2}
We wish to point out that the proof contained in this section is self-contained and is essentially independent of other parts of the paper. In fact, Definition~\ref{d.uhsequence}, Section~\ref{ss:uh_to_resolvent},  and Section~\ref{sss:resolvent_to_uh_2} together could provide a $5$ page complete proof of Theorem~\ref{t.uniformhers}. On the other hand, deep analysis of the uniformly hyperbolic $\mathrm{SL}(2,\R)$ sequences and cocycles may provide more insights regarding the dynamics behind the ergodic type of Schr\"odinger operators.

Fix a $E\in\rho(H_v)$. First,we perform a Combes-Thomas type of estimate concerning the exponential decay of the Green's Function.

Define $M_\b$ to be the multiplication operator $(M_\b\psi)(n)=e^{\b n}\psi_n$. Without loss of generality, we may assume $|\beta|\le 1$. A direct computation shows that
$$
M_{-\b}(H_v-E)M_\b=H_v-E+(e^\b-1)T+(e^{-\b}-1)T^{-1}=H_v-E+B,
$$
where again $T$ is the left shift. The operator $B$ is bounded on $\ell^2(\Z)$ and 
$$
\|B\|\le |(e^\b-1)|+|(e^{-\b}-1)|\le C|\b|.
$$
Clearly, $\|(H_v-E)^{-1}B\|\le \frac12$ if $\b\le \|(H_v-E)^{-1}\|^{-1}/(2C)$. Then
$$
M_{-\b}(H_v-E)M_\b=H_v-E+B=(H_v-E)[I+(H_v-E)^{-1}B]
$$
is invertible. Moreover 
$$
(M_{-\b}(H_v-E)M_\b)^{-1}=M_{-\b}(H_v-E)^{-1}M_\b=[I+(H_v-E)^{-1}B]^{-1}(H_v-E)^{-1},
$$
which implies
$$
\|M_{-\b}(H_v-E)^{-1}M_\b\|\le 2\|(H_v-E)^{-1}\|:=K.
$$
Hence, it holds for all $p,q\in\Z$ that
\begin{align*}
|\langle\delta_p, M_{-\b}(H_v-E)^{-1}M_\b\delta_q\rangle|&=|\langle M_{-\b}\delta_p, (H_v-E)^{-1}M_\b\delta_q\rangle|\\
&=|(H_v-E)^{-1}(p,q)|e^{-\b(p-q)}\\
&\le K
\end{align*}
which gives exponential decay of the Green's Function:
\beq\label{eq:green_decay}
	|(H_v-E)^{-1}(p,q)|\le K e^{-\b|p-q|}.
	\eeq
Those estimates above are  known as the Combes-Thomas type of estimates \cite{combesthomas}.

Let $g_j(n)=(H_v-E)^{-1}(n,j)$. Then $(g_j(n))_{n\in\Z}$ is the unique solution of the equation 
\beq\label{eq:g}
(H_v-E)g_j=\delta_j,
\eeq
where $\delta_j$ is the vector that $\delta_j(m)=1$ if $m=j$ and $0$ otherwise. By \eqref{eq:green_decay}, it holds that
\beq\label{eq:decay_g_j}
|g_j(n)|<Ke^{-\b|n-j|},\mbox{ for all } n,j\in\Z.
\eeq

For each $j\in\Z$, we define $\vec v(j)$ and $\vec w(j)\in\R^2$ so that
$$
\vec v(j)=\binom{g_{j-1}(j)}{g_{j-1}(j-1)}\mbox{ and } \vec w(j)=\binom{g_{j}(j)}{g_j(j-1)}.
$$
Note that $\|\vec v(j)\|\le 2K$ and $\|\vec w(j)\|\le 2K$ for all $j\in\Z$. By \eqref{eq:g}, It holds for each $j\in\Z$ that
\beq\label {eq:middle}
A^{(E-v)}(j-1)\binom{g_{j-1}(j-1)}{g_{j-1}(j-2)}=\binom{g_{j-1}(j)-1}{g_{j-1}(j-1)}.
\eeq
Recall, we assumed that $\|A\|<M$ for all $\mathrm{SL}(2,\R)$ matrices $A$ in question. Thus, \eqref{eq:middle} implies that for some constant $C=C(M)$, it holds that either
$$
C^{-1}\le \|\vec v(j)\|=\left\|\binom{g_{j-1}(j)}{g_{j-1}(j-1)}\right\|\le 2K,
$$
or $C^{-1}<|g_{j-1}(j-2)|$ which in turn implies that
$$
C^{-1}\le \|\vec v(j-1)\|=\left\|\binom{g_{j-2}(j-1)}{g_{j-2}(j-2)}\right\|\le 2K.
$$
Same argument yields similar estimates for $\|\vec w(j)\|$ or $\|\vec w(j+1)\|$.

Now for each $j\in\Z$, we define $\vec s(j)$ so that 
$$
\vec s(j)=
\begin{cases}\vec v(j), &\mbox{ if } \|\vec v(j)\|>C^{-1}\\ 
	A^{(E-v)}(j-1)\vec v(j-1),&\mbox{ otherwise.}
\end{cases}
$$
Similarly, we define $\vec u(j)$ so that 
$$
\vec u(j)=
\begin{cases}\vec w(j), &\mbox{ if } \|\vec w(j)\|>C^{-1}\\ 
A^{(E-v)}(j)^{-1}\vec w(j+1),&\mbox{ otherwise.}
\end{cases}
$$

Thus for all $j\in\Z$, we have 
\beq\label{eq:large_inv_vec}
\|\vec s(j)\|>C^{-1} \mbox{ and }  \|\vec u(j)\|>C^{-1}.
\eeq
In particular, $\vec s(j)\neq \binom{0}0$ and $\vec u(j)\neq \binom00$. Thus for each $j\in\Z$, we may define $s(j)\in\R\PP^1$ to be the direction of $\vec s(j)$ and $u(j)$ be the one of $\vec u(j)$. Then the desired result of this section is a consequence of the following lemma.

\begin{lemma}
$s, u:\Z\to\R\PP^1$ are the stable and unstable directions for $A^{(E-v)}$ as in Definition~\ref{d.uhsequence}. In particular, $A^{(E-v)}\in\CU\CH$.
	\end{lemma}

\begin{proof}
	First we show invariance. We first consider the stable direction $s(j)$. A direct computation shows that 
	$$
	A^{(E-v)}(j)\vec s(j)=\binom{g_{p}(j+1)}{g_{p}(j+2)},
	$$
where $p=j-1$ if $\vec s(j)=\vec v(j)$ and $p=j-2$ otherwise. In both cases, the right-hand side of the equality above must be linearly dependent with $\vec s(j+1)$. Indeed, in all cases and by the fact $A^{(E-v)}(j)\in\mathrm{SL}(2,\R)$,  it must holds for all $n>j$ that 
\begin{align*}
\det [A^{(E-v)}(j)\vec s(j), \vec s(j+1)]
&=\det \left[A^{(E-v)}_{n-j-1}(j+1)\cdot \left(A^{(E-v)}(j)\vec s(j), \vec s(j+1)\right)\right]\\
&=\det \left[A^{(E-v)}_{n-j-1}(j+1)\cdot \begin{pmatrix}g_p(j+1), & g_q(j+1) \\ g_p(j+2), & g_q(j+2)\end{pmatrix}\right]\\
&=\det \left[A_{n-j-1}^{(E-v)}(j+1)\binom{g_p(j+1)}{g_p(j+2)}, A^{(E-v)}_{n-j-1}(j+1)\binom{g_p(j+1)}{g_p(j+2)}\right]\\
&=\det\begin{pmatrix}g_{p}(n), & g_{q}(n)\\ g_{p}(n+1), & g_{q}(n+1)\end{pmatrix},
\end{align*}
 where $p=j-1$ or $j-2$ and $q=j-1$ or $j$. By \eqref{eq:decay_g_j}, the last determinant clearly goes to $0$ as $n\to\infty$.  This implies that the first determinannt is $0$ as it is a constant independent of $n$. This implies that $A^{(E-v)}(j)\vec s(j)$ and $\vec s(j+1)$ are linearly dependent. In other words, $A^{(E-v)}\cdot s(j)=s(j+1)$ which is nothing other than the invariance of $s:\Z\to\R\PP^1$. Simiarly, by letting $n\to-\infty$, we see that $A^{(E-v)}(j-1)^{-1}\vec u(j)$ and $\vec u(j-1)$ are linearly dependent which implies the invariance of the direction $u:\Z\to\R\PP^1$.

Next, we show that exponential decay. Again, it suffices to consider the stable direction $s(j)$ as the argument for $u(j)$ is completely analogous. Bascially in the end of the proof, instead of letting $n\to\infty$, one just need to consider $n\to-\infty$. 

By \eqref{eq:decay_g_j} and \eqref{eq:large_inv_vec}, it holds uniformly for all $j\in\Z$ and all $n\ge 1$ that
$$
\left\|A^{(E-v)}_n(j)\frac{\vec s(j)}{\|\vec s(j)\|}\right\|=\frac1{\|\vec s(j)\|}\left\|\binom{g_{p}(j+n)}{g_{p}(j+n-1)}\right\|<2KCe^{-\beta n},
$$
where $p=j-1$ or $j-2$. This concludes the proof as $M$ and $K$ are independent of $j$ and $n$.
\end{proof}
Note that the estimate also shows that the decaying rate is closely related to $\b$ which is of the order $\|(H_v-E)^{-1}\|^{-1}=\mbox{dist}(E,\sigma(H_v))$. Note also, by Remark~\ref{r.uneqs}, it is automatically true that $s(j)\neq u(j)$ for all $j\in\Z$.

\bigskip

\subsection{Potentials defined dynamically}

Now, we are ready to deduce Theorem~\ref{t.uniformhers2}. Let us start with the following enhanced version of Lemma~\ref{l.finitesae}.

\begin{lemma}\label{l.uniformfsae}
For all $v\in[-M,M]^\Z$, $E\in\sigma(H_v)$ if and only if for each $\varepsilon>0$, there exists a $L=L(M,\varepsilon)$ so that the following holds true. There exists an unit vector $\psi\in\ell^2(\Z)$ supported on an interval $I\subset\Z$ with $|I|\le L$ so that $\|(H_v-E)\psi\|<\varepsilon$.
\end{lemma}

\begin{proof}
By Lemma~\ref{l.finitesae}, we only need to show the only if part. In fact, we only need to show that $L=L(M,\varepsilon)$ is independent of $(v,E)\in[-M,M]^\Z\times\sigma(H_v)\subset [-M,M]^\Z\times[-M-2,M+2]$.

Assume the above mentioned fact is false.  Then there exists an $\varepsilon>0$ with the following property. For each $l\in\Z^+$, there exists a $(v^l,E^l)$, $E^l\in\sigma(H_{v^l})$ such that if any unit vector $\psi\in\ell^2(\Z)$ satisfies $\|(H_{v^l}-E^l)\psi\|<\varepsilon$, then $\psi$ is not supported on any interval $I\subset\Z$ of length less than or equal to $l$.

Let $\Omega_l=\mbox{Hull}(v_l)$ as usual. Since the orbit of $v_l$ is dense in $\Omega_l$, by a standard continuity argument, one may see that for each $\omega\in\Omega_l$, any unit vector $\psi$ satisfying $\|(H_\omega-E^l)\psi\|<\varepsilon$ cannot be supported on any interval $I\subset\Z$ of length less than or equal to $l$. 

On the other hand, $E_l\in\sigma(H_{v^l})$ implies that $A^{(E^l-v^l)}\notin\CU\CH$, hence $(T,A^{(E^l-f)})\notin\CU\CH$. By the same argument of Section~\ref{ss:NotUHtoSpectrum}, for each $l\in\Z^+$, there exists a $\omega^l\in\Omega_l$ and $\psi^l\in\ell^\infty(\Z)$ with $\|\psi^l\|_\infty\le1$ such that $(H_{\omega^l}-E^l)\psi^l=0$. From the construction of $\psi^l$, it holds that $\left\|\binom{\psi^l_0}{\psi^l_{-1}}\right\|=1$. Hence, shifting both $\omega^l$ and $\psi^l$ if necessary and rescaling $\psi^l$, we may assume for all $l\in\Z_+$ it holds that
$$
\psi^l_0=1\mbox{ and }\|\psi^l\|_\infty<C.
$$
By compactness, we may assume
$$
\lim_{l\rightarrow\infty}(\omega^l,\psi^l)=(\omega,\psi)\in[-M,M]^\Z\times[-1,1]^\Z\mbox{ and } \lim_{l\rightarrow\infty}E_l=E\in[-M-2,M+2],
$$
where the convergence of $(\omega^l,\psi^l)$ to $(\omega,\psi)$ is under the product topology. Thus, we must have
$$
(H_\omega-E)\psi=0,\ \psi_0=1,\mbox{ and } \|\psi\|_\infty<C.
$$
By Theorem~\ref{t.uniformhers}, or rather the proof contained in Section~\ref{sss:resolvent_to_uh_1}, $E\in\sigma(H_\omega)$. Moreover, we claim the following.
\vskip 2mm

\noindent \emph{For all sufficiently large $l\in\Z^+$, any unit vector $\phi$ satisfying $\|(H_\omega-E)\phi\|<\varepsilon$ cannot be supported on any interval $I\subset\Z$ of length less than or equal to $l$. }
\vskip 2mm

This claim clearly contradicts with the fact $E\in\sigma(H_\omega)$ and Lemma~\ref{l.finitesae}. So the proof will be completed if we can show the claim holds true.

Indeed, if the claim is not ture, then there exists $L$ so that $\|(H_\omega-E)\phi\|<\varepsilon$ for some $\phi$ supported on an interval with length less than or equal to $L$. Since $\omega^l$ tends to $\omega$ in product topology and $E_l$ tends to $E$, we must have $\|(H_{\omega^l}-E_l)\phi\|<\varepsilon$ for all $l$ sufficiently large. However, for any $l>L$, the existence of such $\phi$ contradicts with the choice of $\omega^l$ and $E_l$.

\end{proof}

As far as we know, Lemma~\ref{l.uniformfsae} was first stated and used as \cite[Lemma 12]{aviladamanikzhang}. It is particularly useful if one wants to prove some continuity property of the spectrum.

\bigskip

Now, we go back to the scenario of Theorem~\ref{t.uniformhers2}. In other words, we have a compact metric space $(\Omega,d)$ with distance $d$, $T:\Omega\rightarrow\Omega$ a homeomorphism, and $f:\Omega\rightarrow\R$ a continuous function. Abusing the notation lightly, $T$ also denotes the left shift operator on the sequence space. Note $(\Omega,T)$ is said to be \emph{topological transitive} if there is a dense $T$--orbit. $(\Omega,T)$ is said to be \emph{minimal} if each $T$--orbit is dense.

For $\omega\in\Omega$, we consider the Schr\"odinger operator $H_\omega$ defined in \eqref{eq:operator} and the associated Schr\"odinger cocycle  $(T,A^{(E-f)})$ as in \eqref{eq:schrodinger_cocycle}.

\begin{theorem}\label{t.densesdo}
Let $(\Omega,T,f)$ be as above. Then for each $\varepsilon>0$, there exists a $\delta>0$ so that the following holds true. If the orbit $\mathrm{Orb}(\omega_0)=\{T^n(\omega_0),\ n\in\Z\}$ of some $\omega_0$ satisfies: 
$$
\mathrm{Orb}(\omega_0)\cap B_\delta(\omega)\neq\varnothing
$$
for all $\omega\in\Omega$, where $B_\delta(\omega)$ is the ball of radius $\delta$ around $\omega$ inside $\Omega$. Then for all $\omega\in\Omega$,
$$
\sigma(H_\omega)\subset B_\varepsilon[\sigma(H_{\omega_0})],
$$
where $B_\varepsilon(S)$ is the ball around the set $S\subset\R$ with the usual distance.
\end{theorem}

\begin{proof}
By compactness of $\Omega$, there exists a $M>0$ such that $\|f\|_\infty<M$. Hence, $\sigma(H_\omega)\subset[-M-2, M+2]$ for all $\omega\in\Omega$. Now by Lemma~\ref{l.uniformfsae}, for the given $\varepsilon$, there exists a $L=L(\varepsilon)$ such that the following holds true. For each $\omega\in\Omega$, $E\in\sigma(H_\omega)$ implies that $\|(H_\omega-E)\psi\|<\varepsilon$ for some unit $\psi\in\ell^2(\Z)$ which is supported in an interval with length less than or equal to $L$. Then there exists a $N\in\Z$ such that $T^{-N}\psi$ is supported on a interval around $0$ and 
$$
\|(H_{T^N\omega}-E)(T^{-N}\psi)\|<\varepsilon.
$$
Then by uniform continuity of $f$, there exists a $\delta>0$, independent of $\omega$, so that the following holds true: if $d(\omega',T^N\omega)<\delta$, then
$
\|(H_{\omega'}-E)(T^{-N}\psi)\|<\varepsilon.
$
In particular, there is some $n\in\Z$ so that $d(T^n\omega_0,T^N\omega)<\delta$ which in turn implies that 
$$
\|(H_{T^n\omega_0}-E)(T^{-N}\psi)\|<\varepsilon.
$$
Thus, we must have $E\in B_\varepsilon[\sigma(H_{T^n\omega_0})]$. Indeed, if $E\in\sigma(H_{T^n\omega_0})$, we are done. Otherwise, it is straightforward to see that the above inequality implies that
$$
\|(H_{T^n\omega_0}-E)^{-1}\|>1/\varepsilon.
$$
Let $g:\sigma(H_{T^n\omega_0})\rightarrow\R$ be the identity function on $\sigma(H_{T^n\omega_0})$. Then the continuous functional calculus implies
$$
\|(g-E)^{-1}\|_\infty>1/\varepsilon,
$$
which implies that $E\in B_\varepsilon[\sigma(H_{T^n\omega_0})]$. It is a standard fact that $H_{\omega_0}$ and $H_{T^n\omega_0}$ are unitary equivalent. Hence $\sigma(H_{\omega_0})=\sigma(H_{T^n\omega_0})$ and $E\in B_\varepsilon[\sigma(H_{\omega_0})]$, concluding the proof.
\end{proof}

With all the preparations, the proof of Theorem~\ref{t.uniformhers2} is now just a few lines.

\begin{proof}[Proof of Theorem~\ref{t.uniformhers2}]
Recall $\overline{\mathrm{Orb}(\omega_0)}=\Omega$ and $\Sigma=\sigma(H_{\omega_0})$. Hence, Theorem~\ref{t.densesdo} implies that $\sigma(H_\omega)\subset B_\e(\Sigma)$ for all $\e>0$ and all $\omega\in\Omega$, which in turn implies that $\sigma(H_\omega)\subset\Sigma$ for all $\omega\in\Omega$. Let
$$
A:\Z\rightarrow\mathrm{SL}(2,\R),\ A^E(n)=A^{(E-f)}(T^n\omega_0).
$$
Then, the fact $\overline{\mathrm{Orb}(\omega_0)}=\Omega$ and Theorem~\ref{t.uiformeg} together clearly imply that
$$
A^E\in\CU\CH\Longleftrightarrow (T,A^{(E-f)})\in\CU\CH.
$$
Hence , $\Sigma=\{E:(T,A^E)\notin\CU\CH\}=\{E:(T,A^{(E-f)})\notin\CU\CH\}$, where the first equality follows from Theorem~\ref{t.uniformhers}.
\end{proof}

The following corollary is an immediate consequence of Theorem~\ref{t.uniformhers2}.

\begin{corollary}\label{c.uniformhers3}
Let $(\Omega,T,f)$ be as in Theorem~\ref{t.uniformhers2}. Assume in addition that $(\Omega,T)$ is minimal, then $\sigma(H_\omega)$ is independent of $\omega\in\Omega$. Let $\Sigma$ denotes the common spectrum. Then we have
$$
\Sigma=\{E:(T,A^{(E-f)})\notin\CU\CH\}.
$$
\end{corollary}
\begin{remark}
	To deduce Theorem~\ref{t.uniformhers2} and Corollary~\ref{c.uniformhers3}, we actually do not really use the full strength of Theorem~\ref{t.densesdo}. Concretely, to obtain $\sigma(H_\omega)\subset\sigma(H_{\omega_0})$ in the proof of Corollary~\ref{t.uniformhers2}, we do not really need the fact that the $\delta$ in the statement of Theorem~\ref{t.densesdo} is independent of $\omega$. If we allow $\delta$ to be dependent on $\omega\in\Omega$, then from the proof one may easily see that there is no need to involve Lemma~\ref{l.uniformfsae}. In fact, Lemma~\ref{l.finitesae} would suffice. However, we wish to provide Lemma~\ref{l.uniformfsae} and the stronger version of Theorem~\ref{t.densesdo} as it may be of independent interest.
	\end{remark}

\section{Avlanche Principle and  Uniformly Hyperbolic Sequence}\label{s:AP}

In this section, we prove Theorem~\ref{t:APNew}. Like the proof of Theorem~\ref{t.uiformeg}, the asymptotic stable and unstable directions play key roles. Basically, we are going to show that under conditions~\eqref{condition-AP3} and \eqref{condition-AP4}, the $1$-step stable and unstable directions as defined in \eqref{eq:AsympDirections} are separated to a certain distance which leads to uniform hyperbolicity. Then based on this information, one can deduce eventually \eqref{eq:AP}. We first need the following preparations.
\begin{lemma}\label{l:NormToAngle}
Let $D,B\in \SL(2,\R)$ satisfying 
\beq\label{eq:ProdNormLarge}
\|DB\|>C^2 \max\left\{\frac{\|D\|}{\|B\|},\ \frac{\|B\|}{\|D\|}\right\}.
\eeq
Then it holds that 
\beq\label{eq:anglenorm}
c|s(D)-u(B)|<\frac{\|DB\|}{\|D\|\|B\|}<C|s(D)-u(B)|
\eeq
and
\beq\label{eq:anglenorm2}
\left|\frac{\|DB\|}{\|D\|\|B\|}-|\sin[s(D)-u(B)]|\right|<C\left(\min\{\|D\|,\|B\|\}\right)^{-2}
\eeq
	\end{lemma}
\begin{proof}
Recall by \eqref{eq:polardecom}, for any $Q\in\SL(2,\R)$, it holds that 
$$
Q=R_{u(Q)}\begin{pmatrix}\|Q\| &0\\ 0 &\|Q\|^{-1}\end{pmatrix}R_{\frac\pi2-s(Q)}.
$$	
Hence, one has 
$$
DB=R_{u(D)}\begin{pmatrix}\|D\| &0\\ 0 &\|D\|^{-1}\end{pmatrix}R_{\frac\pi2-[s(D)-u(B)]}\begin{pmatrix}\|B\| &0\\ 0 &\|B\|^{-1}\end{pmatrix}R_{\frac\pi2-s(B)}.
$$
Let $\t=s(D)-u(B)$. By the form of $DB$ above, it is clearly that 
$$
R_{-u(D)}\cdot DB\cdot R_{s(B)-\frac\pi2}-\begin{pmatrix}\|D\|\|B\|\sin\t &0\\ 0 &0\end{pmatrix}=\begin{pmatrix}0 &-\frac{\|D\|}{\|B\|}\cos\t\\ \frac{\|B\|}{\|D\|}\cos\t &\frac{\sin\t}{\|D\|\|B\|}\end{pmatrix}.
$$
Since rotation matrices preserve the operator norm, triangle inequality then yields
\beq\label{eq:AngleNorm2}
 \left|\left(\|DB\|-\|D\|\|B\|\cdot|\sin[s(D)-u(B)]|\right)\right|<C\max\left\{\frac{\|D\|}{\|B\|},\ \frac{\|B\|}{\|D\|}\right\}.
\eeq
Given \eqref{eq:ProdNormLarge}, one can easily see that 
$$
c\|D\|\|B\|\cdot \left|\sin[s(D)-u(B)]\right|<\|DB\|<C\|D\|\|B\|\cdot \left|\sin[s(D)-u(B)]\right|.
$$ 
Thus, one obtains \eqref{eq:anglenorm} by the fact that $c|\t|<|\sin\t|<C|\t|$ for all $\t\in\R\PP^1$. 

Divide \eqref{eq:AngleNorm2} by $\|D\|\|B\|$ at both sides, we then obtain
$$
\left|\frac{\|DB\|}{\|D\|\|B\|}-|\sin[s(D)-u(B)]|\right|<C\max\{\|D\|^{-2},\|B\|^{-2}\},
$$
which is nothing other than \eqref{eq:anglenorm2}.
\end{proof}

Apply Lemma~\ref{l:NormToAngle} to the sequence given in Theorem~\ref{t:APNew}, we obtain
\begin{corollary}\label{c:NormToAngle}
	Let $A(j),j\in\Z$ be as in Theorem~\ref{t:APNew}, then it holds for all $j\in\Z$ that
	\beq\label{eq:ap_us_1st_sep}
	|s(j)-u(j)|>c\l^{-\frac12}.
	\eeq
	\end{corollary}
\begin{proof}
	It is straightforward computation to see that \eqref{condition-AP4} implies that
	$$
	\frac{\|A(j+1)A(j)\|}{\|A(j+1)\|\|A(j)\|}\ge \l^{-\frac12}.
	$$
In particular, since $\|A(j)\|\ge \l$ for each $j\in\Z$, we then obtain for all $j$ that
$$
\|A(j+1)A(j)\|\ge \|A(j+1)\|\|A(j)\|\l^{-\frac12}>\max\left\{\frac{\|A(j+1)\|}{\|A(j)\|},\ \frac{\|A(j)\|}{\|A(j+1)\|}\right\}.
$$
Thus the condition of Lemma~\ref{l:NormToAngle} is satisfied which in turn implies
$$
|s(j)-u(j)|=|s(A(j))-u(A(j-1))|>c\frac{\|A(j)A(j-1)\|}{\|A(j)\|\|A(j-1)\|}\ge c\l^{-\frac12}.
$$
\end{proof}

Note $\|A(j)\|>\l>C$ for all $j$ and \eqref{eq:ap_us_1st_sep} says that $|s(j)-u(j)|>c\l^{-\frac12}$ for all $j\in\Z$. Then by choosing $\l$ large, we can clearly have that $\l>\rho(\l^{-\frac12})$, where $\rho$ is from Lemma~\ref{l:inv_cone}. In other words, the conditions of Corollary~\ref{c:SeparationUStoUH} are satisfied for this sequence $A$. Following the proof of Theorem~\ref{t.uiformeg}, we may then obtain the uniform exponential growth of $A$, hence $A\in\CU\CH$. But to get \eqref{eq:AP}, we need quantitative estimates. 

We first need the next two lemmas which are special cases of \cite[Lemmas 3, 4]{wangzhang}. We include the proof for completeness since it is much simpler in this special case.
\begin{lemma}\label{l:NormDirectionControl1}
	Let $E=E_2E_1\in\mathrm{SL}(2,\mathbb R)$ such that $\|E_2\|, \|E_1\|>\l>C$ and $|s(E_2)-u(E_1)|>c\l^{-\frac12}$. Then it holds that
	\begin{align}
\label{eq:NormControl1}
	&\|E\|>c\|E_1\|\|E_2\|\cdot |s(E_2)-u(E_1)|>c\l^{\frac32},\\
	\label{eq:SControl1}
&|s(E_1)-s(E)|<C\|E_1\|^{-2}\cdot |s(E_2)-u(E_1)|^{-1},\\
	\label{eq:UControl1}
&|u(E_2)-u(E)|<C\|E_2\|^{-2}\cdot |s(E_2)-u(E_1)|^{-1}.
	\end{align}
	In particular, $|s(E_1)-s(E)|<C\l^{-\frac32}$ and $|u(E_2)-u(E)|<C\l^{-\frac32}$.
\end{lemma}
\begin{proof}
	Let $\t=s(E_2)-u(E_1)$. So $|\t|>c\l^{-\frac12}$.  Define $D$ to be
	$$
	D=
	\begin{pmatrix}
	\|E_2\|&0\\0&\|E_2\|^{-1}\end{pmatrix}R_{\frac\pi2-\t}\begin{pmatrix}\|E_1\|&0\\0&\|E_1\|^{-1}\
	\end{pmatrix}.
	$$
 By the conditions given in the lemma and a direct computation, we may see that 
\beq\label{eq:large_D_norm}
\|E\|=\|D\|>c\|E_2\|\|E_1\||\sin\t|>c\|E_2\|\|E_1\|\l^{-\frac12}>c\l^{\frac32},
\eeq
which takes care of \eqref{eq:NormControl1}. 

Let $\vec e=\binom{0}{1}$. By the form of $D$, it is clearly that
$$
\|D\vec e\|\le \|E_2\|\|E_1\|^{-1}|\cos\t|+\|E_2\|^{-1}\|E_1\|^{-1}|\sin\t|.
$$
Let $\gamma=|\frac\pi2-s(D)|$.  Then it holds that
\begin{align*}
|s(E_1)-s(E)|&=\left|s(E_1)-[s(DR_{\frac\pi2-s(E_1)})]\right|\\
&=\left|s(E_1)-\left[s(D)-\left(\frac\pi2-s(E_1)\right)\right]\right|\\
&=\left|s(D)-\frac\pi2\right|\\
&=\gamma.
\end{align*}

For $\b\in\R\PP^1$, let $\vec \beta$ be a unit vector in the direction of $\b$. Then it clear that
	$$
	\vec e=(\cos\gamma) \vec s(D)+(\sin\gamma)\vec s^{\perp}(D),
	$$ 
which implies that
\begin{align*}
\|D\vec e\|=\|(\cos\gamma) D\vec s(D)+(\sin\gamma)D\vec s^{\perp}(D)\|.
\end{align*}

If $\gamma\le \|E_1\|^{-2}$, then $|s(E_1)-s(E)|=\gamma\le \|E_1\|^{-2}$ which implies \eqref{eq:SControl1} as $|s(E_2)-u(E_1)|^{-1}>c$. So we only need to deal with the case where $\gamma>\|E_1\|^{-2}$. Then by \eqref{eq:large_D_norm} and the fact that $\|E_2\|>\l$, it holds that
\begin{align*}
|(\sin\gamma)D\vec s^{\perp}(D)|&>\|E_1\|^{-2}\|D\|\\
&>c\frac{\|E_2\|}{\|E_1\|}\l^{-\frac12}\\
&>c\frac{\l^{\frac32}}{\|E_1\|\|E_2\|}\\
&>\frac{c\l}{\|E_1\|\|E_2\|\l^{-\frac12}}\\
&>c\l \|D\|^{-1}\\
&>C|(\cos\gamma) D\vec s(D)|.
\end{align*}
Note $|\sin\t|>c|\t|>c\l^{-\frac12}>0$. Combining the above estimates together, we then obtain
\begin{align*}
|\gamma|&<C|\sin\gamma|\\
&=\frac{C}{\|E\|}|(\sin\gamma)D\vec s^{\perp}(D)|\\
&<\frac{C}{\|E\|}\|D\vec e\|\\
&<\frac{C}{\|E_1\|\|E_2\||\sin\t|}(\|E_2\|\|E_1\|^{-1}|\cos\t|+\|E_2\|^{-1}\|E_1\|^{-1}|\sin\t|)\\
&<\frac{C}{\|E_1\|^2|\sin\t|}+\frac{C}{\|E_1\|^2\|E_2\|^2}\\
&<\frac{C}{\|E_1\|^2|\sin\t|}\\
&<C\|E_1\|^{-2}|s(E_2)-u(E_1)|^{-1},
\end{align*}
which is nothing other than  \eqref{eq:SControl1}  since $\gamma=|s(E_1)-s(E)|$. Apply the same argument above to $E^{-1}$ and $D^{-1}$, one then obtain \eqref{eq:UControl1}, concluding the proof.
\end{proof}

The following lemma push the estimates in Lemma~\ref{l:NormDirectionControl1} to all $n\ge 2$.
\begin{lemma}\label{l:NormDirectionControln}
	Let $A(j),\ j\in\Z$ be as in Theorem~\ref{t:APNew}. Then it holds for each $j\in\Z$ and each $2\le n\in\Z_+$ that 
	\begin{align}
	\label{eq:NormControl} \|A_n(j)\|&\ge c\l^{\frac{n+1}{2}},\\
	\label{eq:SControl}|s_n(j)-s_{n-1}(j)|&<C\l^{-(n-1)},\\
		\label{eq:UControl}|u_n(j)-u_{n-1}(j)|&<C\l^{-(n-1)}.
	\end{align}
	\end{lemma}
\begin{proof}
We proceed by induction on $n$. Note for the case $n=2$, \eqref{eq:NormControl} and \eqref{eq:SControl} follow from \eqref{eq:NormControl1} and \eqref{eq:SControl1} by setting $E_1=A(j)$ and $E_2=A(j+1)$ and the fact $C\l^{-\frac32}<C\l^{-1}$. Similarly, \eqref{eq:UControl} follows from \eqref{eq:UControl1} if we set $E_1=A(j-2)$ and $E_2=A(j-1)$. 

Assuming that \eqref{eq:NormControl}-\eqref{eq:SControl} hold true for all $n=2,\ldots, k$ and all $j\in\Z$.  Then we want to move to the case $n=k+1$. First, it holds that
\begin{align}
\nonumber |u_{k}(j+k)-u_1(j+k)|&\le \sum^{k-1}_{l=1}|u_{l+1}(j+k)-u_l(j+k)|\\
\nonumber &\le C\l^{-\frac32}+\sum^{k-1}_{l=2}C\l^{-l}\\
\label{eq:AngleControl0}&\le C\l^{-\frac32}.
\end{align}
Note $u_1(j+k)=u(j+k)$. Consequently, by Corollary~\ref{c:NormToAngle} it holds that
\begin{align}
\nonumber |s(j+k)-u_k(j+k)|
&=|s(j+k)-u(j+k)+u(j+k)-u_k(j+k)|\\
\nonumber &\ge  |s(j+k)-u(j+k)|-|u(j+k)-u_k(j+k)|\\
\nonumber &\ge c\l^{-\frac12}-C\l^{-\frac32}\\
\label{eq:AngleControl}&\ge c\l^{-\frac12}.
\end{align}
Similarly, by the same argument of \eqref{eq:AngleControl0}, it holds that $|s_k(j-k-1)-s(j-k-1)|\le C\l^{-\frac32}$. Together with Corollary~\ref{c:NormToAngle}, we then obtain
\begin{align}
\nonumber |s_k(j-k-1)-u(j-k-2)|
&=|s_k(j-k-1)-s(j-k-1)+s(j-k-1)-u(j-k-2)|\\
\nonumber &\ge |s(j-k-1)-u(j-k-2)|-|s_k(j-k-1)-s(j-k-1)|\\
\nonumber &\ge c\l^{-\frac12}-C\l^{-\frac32}\\
\label{eq:AngleControl2}&\ge c\l^{-\frac12}.
\end{align}

Thus, we may apply Lemma~\ref{l:NormDirectionControl1} with $E_1=A_k(j)$ and $E_2=A(j+k)$ and get that
\begin{align*}
\|A_{k+1}(j)\|
&=\|A(k+j)A_k(j)\|\\
&>c\|A(k+j)\|\|A_k(j)\|\l^{-\frac12}\\
&>c\l\cdot\l^{\frac{k+1}{2}}\l^{-\frac12}\\
&=c\l^{\frac{k+2}{2}},
\end{align*}
which takes care of \eqref{eq:NormControl} for $n=k+1$.  Next, combine \eqref{eq:AngleControl} and \eqref{eq:SControl1} with $E_1=A_k(j)$ and $E_2=A(k+j)$, we obtain
\begin{align}
\nonumber |s_{k+1}(j)-s_k(j)|&=|s[A(k+j)A_k(j)]-s(A_k(j))|\\
\nonumber &<C\|A_k(j)\|^{-2}|\cdot |s(j+k)-u_k(j+k)|^{-1}\\
\nonumber &<C\l^{-(k+1)}\l^{\frac12}\\
&<C\l^{-k},
\end{align}
which clearly takes care of the \eqref{eq:SControl} for $n=k+1$. On the other hand, combine \eqref{eq:AngleControl} and \eqref{eq:UControl1} with $E_1=A(k-j-2)$ and $E_2=A_k(j-k-1)$, we obtain
\begin{align}
\nonumber |u_{k+1}(j)-u_k(j)|&=|u[A_k(j-k-1)A(j-k-2)]-u(A_k(j-k-1))|\\
\nonumber &<C\|A_k(j-k-1)\|^{-2}|\cdot |s_k(j-k-1)-u(j-k-2)|^{-1}\\
\nonumber &<C\l^{-(k+1)}\l^{\frac12}\\
&<C\l^{-k},
\end{align}
which takes of  \eqref{eq:UControl} for step $n=k+1$, concluding the proof.
	\end{proof}

Now, we are ready to prove Theorem~\ref{t:APNew}.
\begin{proof}[Proof of Theorem~\ref{t:APNew}]
By the discussion following Corollary~\ref{c:NormToAngle}, we already know that $\{A(j),\ j\in\Z\}$ is uniformly hyperbolic. In fact, we have more precise estimate. By \eqref{eq:NormControl}, it holds for all $j\in\Z$ and all $n\ge 1$ that
$$
\|A_n(j)\|>c\l^{\frac{n+1}{2}}>c(\sqrt\l)^{n}.
$$

For the proof of \eqref{eq:AP}, we first note it holds for all $j\in\Z$ and $n\in\Z_+$ that
\beq\label{eq:AP1}
\log\|A_n(j)\|=\log\|A(j+n-1)\|+\log\|A_{n-1}(j)\|+\log\frac{\|A(j+n-1)A_{n-1}(j)\|}{\|A(j+n-1)\|\|A_{n-1}(j)\|}.
\eeq
 We may apply \eqref{eq:AP1} to $\log\|A_{n-1}(j)\|$ and rewrite \eqref{eq:AP1} as 
 \beq\label{eq:AP2}
 \log\|A_n(j)\|=\log\|A_{n-2}(j)\|+\sum^{n-1}_{k=n-2}\log\|A(j+k)\|+\sum^{n-1}_{k=n-2}
 \log\frac{\|A(j+k)A_{k}(j)\|}{\|A(j+k)\|\|A_{k}(j)\|}.
 \eeq
 Apply this process repeatedly to $A_n(j), A_{n-1}(j),\ldots, A_2(j)$, we then obtain
\beq\label{eq:PreAP}
\log\|A_n(j)\|=\sum^{n-1}_{k=0}\log\|A(j+k)\|+\sum^{n-1}_{k=1}\log\frac{\|A(j+k)A_{k}(j)\|}{\|A(j+k)\|\cdot\|A_{k}(j)\|}.
\eeq
Now for each $k\ge1$, by the proof of Corollary~\ref{c:NormToAngle}, condition of Lemma~\ref{l:NormToAngle} is satisfied for the pair $A(j+k)$ and $A(j+k-1)$. Similarly, \eqref{eq:AngleControl} implies the condition is satisfied for $A(j+k)$ and $A_k(j)$ as well. Thus, applying \eqref{eq:anglenorm2} to both pairs, we obtain
$$
\left|\frac{\|A(j+k)A_{k}(j)\|}{\|A(j+k)\|\cdot\|A_{k}(j)\|}-|\sin[s(j+k)-u_k(j+k)]|\right|<C\l^{-2}
$$
and 
$$
\left|\frac{\|A(j+k)A(j+k-1)\|}{\|A(j+k)\|\cdot\|A(j+k-1)\|}-|\sin[s(j+k)-u(j+k)]|\right|<C\l^{-2}.
$$
On the other hand, it holds that
\begin{align*}
&\left| (|\sin[s(j+k)-u_k(j+k)]|- |\sin[s(j+k)-u(j+k)]|)\right|\\
&<\left| \sin[s(j+k)-u_k(j+k)]- \sin[s(j+k)-u(j+k)]\right|\\
&<\left|[s(j+k)-u_k(j+k)]- [s(j+k)-u(j+k)]\right|\\
&=|u_k(j+k)-u(j+k)|\\
&\le C\l^{-\frac32},
\end{align*}
where the last inequality follows from \eqref{eq:AngleControl0}. Combine the three inequalities above, we then obtain
$$
\left|\frac{\|A(j+k)A_{k}(j)\|}{\|A(j+k)\|\cdot\|A_{k}(j)\|}-\frac{\|A(j+k)A(j+k-1)\|}{\|A(j+k)\|\cdot\|A(j+k-1)\|}\right|<C\l^{-\frac32}.
$$
Apply \eqref{eq:anglenorm} and Corollary~\ref{c:NormToAngle} to $A(j+k)$ and $A(j+k-1)$, we obtain
$$
\frac{\|A(j+k)A(j+k-1)\|}{\|A(j+k)\|\|A(j+k-1)\|}>c|s(j+k)-u(j+k)|>c\l^{-\frac12}.
$$
It is straightforward calculus type of estimate that 
$$
|\log a -\log b|<C\left|\frac1b(a-b)\right| \mbox{ when } b>C|a-b|. 
$$
Thus for each $k\ge1$, the inequality above implies that
\begin{align*}
&\left|\log\frac{\|A(j+k)A_{k}(j)\|}{\|A(j+k)\|\cdot\|A_{k}(j)\|}-\log\frac{\|A(j+k)A(j+k-1)\|}{\|A(j+k)\|\cdot\|A(j+k-1)\|}\right|\\
&\le C\frac{\|A(j+k)\|\|A(j+k-1)\|}{\|A(j+k)A(j+k-1)\|}\cdot\left|\frac{\|A(j+k)A_{k}(j)\|}{\|A(j+k)\|\cdot\|A_{k}(j)\|}-\frac{\|A(j+k)A(j+k-1)\|}{\|A(j+k)\|\cdot\|A(j+k-1)\|}\right|\\
&\le
C\l^{\frac12}\l^{-\frac32}\\
&=C\l^{-1}.
\end{align*}
Combine \eqref{eq:PreAP} and the estimate above, we then obtain
\begin{align*}
&\left|\log\|A_n(j)\|-\sum^{n-1}_{k=0}\log\|A(j+k)\|-\sum^{n-1}_{k=1}\log\frac{\|A(j+k)A(j+k-1)\|}{\|A(j+k)\|\cdot\|A(j+k-1)\|}\right|\\
&=\left|\sum^{n-1}_{k=1}\log\frac{\|A(j+k)A_{k}(j)\|}{\|A(j+k)\|\cdot\|A_{k}(j)\|}-\sum^{n-1}_{k=1}\log\frac{\|A(j+k)A(j+k-1)\|}{\|A(j+k)\|\cdot\|A(j+k-1)\|}\right|\\
&\le \sum^{n-1}_{k=1}\left|\log\frac{\|A(j+k)A_{k}(j)\|}{\|A(j+k)\|\cdot\|A_{k}(j)\|}-\log\frac{\|A(j+k)A(j+k-1)\|}{\|A(j+k)\|\cdot\|A(j+k-1)\|}\right|\\
&\le C\frac{n-1}{\l}\\
&\le C\frac{n}{\l}.
\end{align*}
A direct computation shows that the first line in the estimate above is nothing other than 
$$
\left|\log\|A_n(j)\|+\sum_{k=1}^{n-2}\log\|A(j+k)\|-\sum_{k=0}^{n-2}\log\|A(j+k+1)A(j+k)\|\right|,
$$
concluding the proof.
\end{proof}	
\bigskip

\noindent\textit{\bf Acknowledgments.} Part of the works were carried out when the author was a graduate student at Northwestern. He wishes to thank his thesis advisor Amie Wilkinson for her support and many helpful discussions. He is deeply gratefuly to his co-advisor Artur Avila, who generously shared with the author many of the new ideas contained in this paper. In particular, many parts of the proofs of Lemma~\ref{l.distanceus}, Theorem~\ref{l.nontrivialbo}, and Theorem~\ref{t.densesdo} were suggested to the author by him. The author would like to thank Wilhelm Schlag for providing him with the whole Section~\ref{sss:resolvent_to_uh_2}, for reading some parts of the preprint, and for some useful conversations. The author also thanks Adam Black for finding a mistake in the proof of Lemma~\ref{l.existenceus}, thanks Yakir Forman for providing a correction of it, and thanks both of them for carefully reading the paper. Finally, the author wishes to thank the referees for many helpful comments which help improve the presentation of the paper greatly.

\end{document}